\documentclass[12pt,cd, reqno]{amsart}
\usepackage{amsfonts, amssymb,amsmath, amsthm, amsxtra, latexsym, mathrsfs, amscd}
\usepackage{amsmath, amssymb, amsthm, times, float, pifont, eufrak}
\usepackage{graphics}
\usepackage{stackrel}
\usepackage{hyperref}
\usepackage{tikz-cd}

\usepackage{verbatim}
\usepackage{dynkin-diagrams}
\usepackage{listings}

\usepackage[all]{xy}
\usepackage{ytableau}
\usepackage{stackrel}
\usepackage{pdfpages}
\usepackage[latin1]{inputenc} % 

\newtheorem{theo+}{Theorem}[section]
\newtheorem{prop+}[theo+]{Proposition}
\newtheorem{coro+}[theo+]{Corollary}
\newtheorem{lemm+} [theo+]{Lemma}
\newtheorem{deep+}  [theo+]  {Deep Result}
\newtheorem{fact+}  [theo+]  {Fact}
\theoremstyle{definition}
\newtheorem{exam+}  [theo+]  {Example}
\newtheorem{rema+}  [theo+]  {Remark}
\newtheorem{defi+}  [theo+]  {Definition}
\newtheorem{xca+}[theo+]{Exercise}
\newenvironment{theorem}{\begin{theo+}}{\end{theo+}}

\newenvironment{corollary}{\begin{coro+}}{\end{coro+}}
\newenvironment{lemma}{\begin{lemm+}}{\end{lemm+}}

\newenvironment{remark}{\begin{rema+}}{\end{rema+}}

\numberwithin{equation}{section}
\newtheorem{example}{Example}

\usepackage{color}

\bigskip

\def\refn#1.#2{\expandafter\def\csname#1\endcsname{[#2]}}
\def\refnr#1.{\csname#1\endcsname}

\def\eshd{\endsubhead}

\def\del{\delta}
\def\Del{\Delta}
\def\fpw{\frac{\partial\,}{\partial w}}

\def\la{\lambda}
\def\La{\Lambda}
\def\bbar{\,|\,}

\def\fa{\mathfrak a}
\def\fb{\mathfrak b}
\def\fe{\mathfrak e}
\def\fg{\mathfrak g}
\def\fgl{\mathfrak {gl}}
\def\fk{\mathfrak k}
\def\fh{\mathfrak h}
\def\fl{\mathfrak l}
\def\fm{\mathfrak m}
\def\fn{\mathfrak n}
\def\fp{\mathfrak p}
\def\fq{\mathfrak q}
\def\fs{\mathfrak s}
\def\fspin{\mathfrak {spin}}
\def\ft{\mathfrak t}
\def\fsu{\mathfrak {su}}
\def\fsl{\mathfrak {sl}}
\def\fspin{\mathfrak {spin}}
\def\fsp{\mathfrak {sp}}
\def\fso{\mathfrak {so}}
\def\fv{\mathfrak v}
\def\a{\alpha}
\def\b{\beta}
\def\epsi{\epsilon}
\def\Claminv2{|C(\Lambda)|^{-2}}
\def\ga{\gamma}
\def\Ga{\Gamma}
\def\varepsi{\varepsilon}
\def\lam{\lambda}
\def\Lam{\Lambda}
\def\ome{\omega}
\def\Ome{\Omega}

\def\dmua{d\mu_\a}
\def\dmub{d\mu_\beta}
\def\dmuabp{d\mu_{\a+\beta+p}}
\def\de{d\varepsilon}
\def\K#1#2#3{K^{#1}(#2,#3)}
\def\bK#1#2#3{\bar K^{#1}(#2,#3)}
\def\LtvDmua{L^2(D,d\mu_\a)}
\def\U#1#2{U_{#1}^{(#2)}}
\def\bU#1#2{\bar U_{#1}^{(#2)}}

\def\l{\left}
 \def\({\left(}
\def\r{\right}  
 \def\){\right)}
\def\Aa2D{A^{\a,2}(D)}
\def\bAa2D{\overline{A^{\a,2}(D)}}
\def\Ab2D{A^{\beta,2}(D)}
\def\bAb2D{\overline{A^{\beta,2}(D)}}
\def\AtA{\Aa2D \otimes \bAb2D}

\def\abs#1{\vert#1\vert}
\def\innerp#1#2{(#1, #2)}
\def\Innerp#1#2{\left(#1, #2\right)}
\def\ainnerp#1#2{\langle#1, #2\rangle}
\def\norm#1{\Vert#1\Vert}
\def\Norm#1_#2{\Vert#1\Vert_{#2}}
\def\Normsq#1{\Vert#1\Vert^2}
\def\normsq#1{\vert#1\vert^2}
\def\phipl12{\phi_{p_{l_1}, p_{l_2}}}
\def\phip01{\phi_{p_{0}, p_{0}}}

\newcommand\I{\int\limits}
\renewcommand\S{\sum\limits}
\renewcommand\P{\prod\limits}
\def\a{\alpha}
\def\b{\beta}
\def\ga{\gamma}
\def\Claminv2{|C(\Lambda)|^{-2}}
\def\del{\delta}
\def\Del{\Delta}

\def\varepsi{\varepsilon}

\def\sig{\sigma}
\def\Sig{\Sigma}
\def\Ga{\Gamma}
\def\lam{\lambda}
\def\Lam{\Lambda}

\def\ad{\operatorname{ad}}
\def\Ad{\operatorname{Ad}}
\def\Lie{\operatorname{Lie}}

\def\bd{\mathbf}
\def\ch{\operatorname{ch}}
\def\coth{\operatorname{coth}}

\def\det{\operatorname{det}}
\def\Det{\operatorname{Det}}
\def\diag{\operatorname{diag}}
\def\Ind{\operatorname{Ind}}
\def\diag{\operatorname{diag}}
\def\Mod{\operatorname{Mod}}
\def\min{\operatorname{min}}
\def\mult{\operatorname{mult}}
\def\max{\operatorname{max}}
\def\rank{\operatorname{rank}}
\def\tanh{\operatorname{tanh}}
\def\coth{\operatorname{coth}}
\def\ch{\operatorname{ch}}
\def\exp{\operatorname{exp}}
\def\Exp{\operatorname{Exp}}
\def\Pr{\operatorname{Pr}}

\def\Spin{\operatorname{Spin}}
\def\SO{\operatorname{SO}}
\def\sh{\operatorname{sh}}
\def\tanh{\operatorname{tanh}}
\def\sh{\operatorname{sh}}
\def\sgn{\operatorname{sgn}}
\def\tr{\operatorname{tr}}
\def\vol{\operatorname{vol}}
\def\Ker{\operatorname{Ker}}
\def\End{\operatorname{End}}
\def\Hom{\operatorname{Hom}}

\def\dmua{d\mu_\a}
\def\dmub{d\mu_\beta}
\def\dmuabp{d\mu_{\a+\beta+p}}
\def\de{d\varepsilon}
\def\K#1#2#3{K^{#1}(#2,#3)}
\def\bK#1#2#3{\bar K^{#1}(#2,#3)}
\def\LtvDmua{L^2(D,d\mu_\a)}
\def\U#1#2{U_{#1}^{(#2)}}
\def\bU#1#2{\bar U_{#1}^{(#2)}}

\def\Aa2D{A^{\a,2}(D)}
\def\bAa2D{\overline{A^{\a,2}(D)}}
\def\Ab2D{A^{\beta,2}(D)}
\def\bAb2D{\overline{A^{\beta,2}(D)}}
\def\AtA{\Aa2D \otimes \bAb2D}
\def\phipl12{\phi_{p_{l_1}, p_{l_2}}}
\def\phip01{\phi_{p_{0}, p_{0}}}
\def\m{\underline{\mathbf m}}
\def\ov{\overline{v}}
\def\s{\underline{\mathbf s}}
\def\u{\underline{\mathbf u}}
\def\ov{\overline{v}}
\def\intR{\int_{\mathbb R^2}}
\def\intC{\int_{\mathbb C^2}}

\def\m{\underline{\mathbf m}}
\def\n{\underline{\mathbf n}}

\def\za{\aa{}}
\def\zae{\"a{}}
\def\zo{\"o{}}

\def\Za{\AA{}}
\def\Zae{\"A{}}
\def\Zo{\"O{}}

\def\bc{\mathbb C}
\def\br{\mathbb R}
\def\bn{\mathbb N}
\def\bz{\mathbb Z}
\def\bq{\mathbb Q}
\def\bh{\mathbb H}
\def\bo{\mathbb O}
\def\bk{\mathbb K}

\def\alg/{algebra}

\def\Alg/{Algebra} 

\def\alt/{alternative} % \alt/ly
\def\anal/{analytic}
\def\analfunc/{\anal/\ \func/}
\def\Ans/{\it Answer. \normal}
\def\ass/{associative}
\def\nass/{non-\ass/}
\def\autom/{automorphism}
\def\homom/{homomorphism}
\def\isom/{isomorphism}
\def\bdd/{bounded}
\def\Bdd/{Bounded}
\def\bddsymdom/{bounded \sym/ \dom/}
\def\Cartdom/{Cartan \dom/}
\def\bdry/{boundary}
\def\bsd/{\bdd/ \symdom/}
\def\bv/{boundary value}
\def\cf/{{\it cf}\.}
\def\Cf/{{\it Cf}\.}
\def\charr/{character}
\def\coeff/{coefficient}
\def\comm/{commutative}
\def\cpct/{compact}
\def\compl/{complex}
\def\comp/{complex}
\def\Comp/{Complex}
\def\conf/{conformal}
\def\conj/{conjugate}
\def\conn/{connect}
\def\cont/{continuous}
\def\conv/{converge} % \conv/nce \conv/nt
\def\convc/{convergence}
\def\convt/{convergent}
\def\convx/{convex}
\def\coord/{coordinate}
\def\lcoord/{local coordinate}
\def\Corr/{Corresponding}
\def\corr/{corresponding}
\def\corrd/{correspond}
\def\cov/{covariant}
\def\decomp/{decomposition}
\def\deco/{decompose}
\def\diff/{different} % \diff/iable \diff/ial
\def\Diff/{Different} % \Diff/able \Diff/ial
\def\dimn/{dimension} % \dimen/al
\def\distr/{distribution} % \distr/al
\def\div/{diverge} % \div/nt 
\def\dom/{domain}
\def\eg/{\hbox{\it e.g}\.}
\def\eigenf/{eigen\-\func/}
\def\eigensp/{eigen\-space}
\def\eigenv/{eigen\-value}
\def\eq/{equation}
\def\equa/{equation}
\def\de/{\diff/ial \equa/}
\def\do/{\diff/ial operator}
\def\ode/{ordinary \de/}
\def\pde/{partial \de/}
\def\pdo/{partial \diff/ial operator}
\def\psdo/{pseudo \diff/ial operator}
\def\fin/{finite}
\def\Ex/{\it Example.\ \normal}
\def\Exnr#1/{\it Example #1.\ \normal}
\def\foll/{follow}
\def\follg/{following}
\def\Follg/{Following}
\def\func/{function}
\def\Func/{Function}
\def\Fonc/{Fonc\-tion}
\def\fonc/{fonc\-tion}
\def\Funk/{Funk\-tion}
\def\funk/{Funk\-tion}
\def\gen/{general}
\def\har/{harmonic}
\def\Hint/{\it Hint. \normal}
\def\hist/{historic}
\def\histcl/{historical}
\def\hol/{holo\-morphic}
\def\homog/{ho\-mo\-ge\-ne\-ous}
\def\hyp/{hyper\-bolic}
\def\hyperg/{hyper\-geometric}
\def\ie/{\hbox{\it i.e.}}
\def\iff/{if and only if}
\def\ineq/{inequality}
\def\infra/{{\it inf\-ra}}
\def\ultra/{{\it ult\-ra}}
\def\Inpart/{In particular}
\def\inpart/{in particular}
\def\instof/{instead of}
\def\interps/{interpolation space}
\def\interp/{interpolation}
\def\Interp/{Interpolation}
\def\interpr/{Interpretation}
\def\Intr/{Introduction}
\def\intv/{interval}
\def\inv/{invariant}
\def\invc/{invariance}
\def\Iowords/{In other words}
\def\iowords/{in other words}
\def\ipr/{inner product}
\def\irred/{irreducible}
\def\lb/{line bundle}
\def\lin/{linear}
\def\lhs/{left hand side}
\def\rhs/{right hand side}
\def\loc/{local}
\def\math/{mathematic} 
\def\mathcn/{\math/ian}
\def\manif/{manifold}
\def\meas/{measure}
\def\measl/{measurable}
\def\mero/{mero\-morphic}
\def\mon/{monomial}
\def\monog/{monogenic}
\def\multy/{multiply}
\def\multn/{multiplication}
\def\nas/{necessary and sufficient}
\def\nbd/{neighborhood}
\def\neg/{negative}
\def\nondeg/{nondegenerate}
\def\Oohand/{On the other hand}
\def\oohand/{on the other hand}
\def\Oonhand/{On the one hand}
\def\oonhand/{on the one hand}
\def\oper/{operator}
\def\orth/{ortho\-gonal}
\def\orthon/{ortho\-normal}
\def\otoh/{on the other hand}
\def\quat/{quaternion}
\def\pp/{\hbox{a. e.}}
\def\psh/{plurisubharmonic}
\def\pol/{polynomial}
\def\pot/{potential}
\def\pos/{positive}
\def\princ/{principle}
\def\prob/{probability}
\def\proj/{projective}
\def\projn/{projection}
\def\Proof/{\it Proof:\normal}
\def\Rem/{\it Remark\normal}
\def\Remnr#1/{\it Remark\ \normal #1. }
\def\rep/{representation}
\def\reps/{representations}
\def\meta/{metaplectic representation}
\def\repr/{reproducing}
\def\reprker/{reproducing kernel}
\def\resp/{respective} % \resp/ly
\def\resply/{respectively}
\def\restr/{restriction}
\def\sa/{self-adjoint}
\def\st/{such that}
\def\Sec{Section~}
\def\SSec{Subsection~}
\def\sol/{solution}
\def\ru/{space}
\def\sph/{spherical}
\def\ssp/{sub\ru/}
\def\sym/{symmetric}
\def\Sym/{Symmetric}
\def\symb/{symbol}
\def\symbc/{symbolic}
\def\symdom/{\sym/ domain}
\def\symp/{symplectic}
\def\Theor#1/{\fet Theorem #1.\ \normal}
\def\Lem#1/{\fet Lemma #1.\ \normal}
\def\Lemma/{\fet Lemma.\ \normal}
\def\topl/{topology}
\def\topll/{topological}
\def\transf/{transform}
\def\transl/{translation}
\def\transfn/{transformation}
\def\transv/{transvectant}
\def\trig/{trigonometric}
\def\tril/{trilinear}
\def\trilf/{trilinear form}
\def\uhp/{upper halfplane}
\def\uhs/{upper halfspace}
\def\vb/{vector bundle}
\def\vf/{vector field}
\def\vsp/{vector space}
\def\wrt/{with respect to}
\def\Wlog/{Without loss of generality}
\def\a{\alpha}
\def\ba{\beta}
\def\ve{\varepsilon}
\def\e{\varepsilon}
\def\E{\Varepsilon}
\def\f{\varphi}
\def\gam{\gamma}
\def\Gam{\Gamma}
\def\lam{\lambda}
\def\Lam{\Lambda}
\def\om{\omega}
\def\Om{\Omega}
\def\sig{\sigma}
\def\Sig{\Sigma}
\def\th{\theta}
\def\ups{\upsilon}
\def\Ab/{Abel}
\def\Ban/{Banach}
\def\Bansp/{\Ban/ space}
\def\Belt/{Bel\-tra\-mi}
\def\Berg/{Berg\-man}
\def\Bern/{Ber\-nou\-lli}
\def\Berz/{Berezin}
\def\Bess/{Bessel}
\def\Cart/{Car\-tan}
\def\Cay/{Cay\-ley}
\def\CG/{Clebsch-Gordan}
\def\Cl/{Clifford}
\def\CR/{Cauchy-Rie\-mann}
\def\Dir/{Dirichlet}
\def\Eucl/{Euclide}
\def\Eucln/{Euclidean}
\def\F/{Fourier}
\def\Hank/{Hankel}
\def\Hankf/{\Hank/ form}
\def\Herm/{Hermite}
\def\Hilb/{Hilbert}
\def\Hilbs/{Hilbert space}
\def\Hilbsp/{Hilbert space}
\def\HS/{Hilbert-Schmidt}
\def\Lag/{La\-grange}
\def\Lap/{La\-place}
\def\LapBelt/{\Lap/-\Belt/}
\def\Leb/{Lebesgue}
\def\Marc/{Mar\-cin\-kie\-wicz}
\def\Moeb/{Moebius}
\def\Moebt/{Moebius transformation}
\def\Moebtransfn/{Moebius transformation}
\def\Pla/{Plan\-che\-rel}
\def\Poin/{Poin\-car\'e}
\def\Riem/{Rie\-mann}
\def\Riemn/{\Riem/ian}
\def\psRiemn/{pseudo-\Riem/ian}
\def\Riems/{Rie\-mann surface}
\def\Schroe/{Schr\"odinger}
\def\Weier/{Weier\-strass}
\def\re{\operatorname{Re}}
\def\im{\operatorname{Im}}
\def\id{\text{id}}
\def\anal/{analytic}
\def\bsd/{bounded symmetric domain  }
\def\bdd/{bounded}
\def\calc/{calculation}\def\conj{conjugate}
\def\calci/{calculating}\def\eg{e.g.}
\def\conj/{conjugate}
\def\deco/{decomposition}
\def\eg/{e.g.}
\def\fct/{function}
\def\gp/{group}
\def\hw/{highest weight}
\def\hwv/{highest weight vector}
\def\hwvs/{highest weight vectors}
\def\lw/{lowest weight}
\def\lwv/{lowest weight vector}
\def\lwvs/{lowest weight vectors}
\def\hds/{holomorphic discrete series}
\def\iff/{if and only if}
\def\inv/{invariant}
\def\irrde/{irreducible decomposition}
\def\meas/{measure}
\def\transf/{transform}
\def\rep/{representation}
\def\resp/{respectively}
\def\inters/{intertwines}
\def\interg/{intertwining}
\def\meta/{metaplectic representation}
\def\qu/{quaternion}
\def\rep/{representation}
\def\symdom/{ symmetric domain}
\def\st/{such that}
\def\shd/{subhead}
\def\transf/{transform}
\def\wrt/{with respect to}

\def\ra{\rightarrow}
\def\grad{\operatorname{grad}}
\def\supp{\operatorname{supp}}
\def\Norm#1#2#3{\Vert#1\Vert^{#3}_{{#2}}}
\def\bigNorm#1#2#3{\big\Vert#1\big\Vert_{#2}^{#3}}
\def\biggNorm#1#2#3{\bigg\Vert#1\bigg\Vert_{#2}^{#3}}
\def\sprod#1{\langle#1\rangle}
\def\dd#1#2{\dfrac{d#1}{d#2}}
\def\dde#1{\dfrac{d\phantom{#1}}{d#1}}
\def\DD#1#2{\dfrac{\partial#1}{\partial#2}}
\def\DDe#1{\DD{\phantom{#1}}{#1}}
\def\DtD#1#2#3{\dfrac{\partial^2#1}{\partial#2\partial#3}}
\def\DtDe#1#2{\dfrac{\partial^2\phantom{#1#2}}{\partial#1\partial#2}}
\def\DtDt#1#2{\dfrac{\partial^2#1}{\partial#2^2}}
\def\DtDte#1{\dfrac{\partial^2\phantom{#1}}{\partial#1^2}}
\def\DEF{\,\,{{\overset\text{def}\to=}}\,\,}
\def\FED{\,\,{{\overset\text{fed}\to=}}\,\,}
\def\oh{\operatorname{o}}
\def\Oh{\operatorname{O}}
\def\tr{\operatorname{tr}}

\def\threestars{\bigskip\centerline{* * *}\medskip}

\def\abstractname{Abstract}
\def\chrefname{References}

\begin{document}
\title[$K$-types of 
unitary highest weight representations]{
  On the $K$-types of certain families of unitary highest weight representations of
  Hermitian Lie groups 
}
\author{Pavle Pand\v zi\'c and Genkai Zhang}
%%%%
%\email{}
\address[Pand\v zi\'c]{Department of Mathematics, Faculty of Science, University of Zagreb, Bijeni\v cka 30, 10000 Zagreb, Croatia}
\email{pandzic@math.hr}
\address[Zhang]{Mathematical Sciences, Chalmers University of Technology and
Mathematical Sciences, G\"oteborg University, SE-412 96 G\"oteborg,
Sweden.}
\email{genkai@chalmers.se}


\thanks{
Research by P.~Pand\v zi\'c is supported by "Implementation of cutting-edge research and its application as part of the Scientific Center of Excellence for Quantum and Complex Systems, and Representations of Lie Algebras", Grant No. PK.1.1.10.0004, co-financed by the European Union through the European Regional Development Fund - Competitiveness and Cohesion Programme 2021- 2027., by the Croatian Science Foundation (HRZZ), grant no. IP-2025-02-6514, and
by the European Union - NextGenerationEU through the National Recovery and Resilience Plan 2021-2026. Institutional grant of University of Zagreb Faculty of Science (IK IA 1.1.3. Impact4Math). 
 Research by G.~Zhang is partially supported by 
the Swedish
Science Council (VR)
}
%\draft
\begin{abstract}
  Let $(\mathfrak g, \mathfrak k)$ 
  be the complexication of a   Hermitian symmetric 
  pair  and let $\gamma_1 <\cdots<\gamma_r$
  be the Harish-Chandra strongly orthogonal roots, with
 $\gamma_1$ the lowest  non-compact root, and let $\zeta$ be the
  normalized central character of $\fk$.
  We consider the unitary highest weight
  representations $L(\lambda(x))$ of 
$\mathfrak g $
of
  highest weights of the form $\lambda(x)=-k\gamma_1 -x\zeta$, $x\in\mathbb{R}$. In particular, we are interested in the case $x=a$, the last reduction point in
  the analytic continuation in $x$.
 We prove that
  $L(-k\gamma_1-a \zeta
  )$, $k\ge 1$,
   is $K$-multiplicity free and
   we describe
   all its K-types. 
\end{abstract}

\maketitle

\section{Introduction}

Let $(G, K)$ 
be an irreducible Hermitian
symmetric pair
and let $(\mathfrak g,\mathfrak k)
$ be the complexification
of the symmetric
pair $(\fg_0, \fk_0)$
of the Lie algebras of $G$ and $K$.
In the present paper we shall find the $K$-types
for a family of unitary highest weight representations of $\mathfrak g$. 

To describe
our result
let $\gamma_1< \cdots<\gamma_r$ be the Harish-Chandra strongly orthogonal roots
 with $\gamma_1$ the lowest non-compact root. 
The analytic continuation
of the holomorphic discrete series
$L(\lambda +x\zeta), x<<0$ in $x
$ has maximally $r=\rank G/K$ reduction points,
and has precisely
$r=\rank G/K$ points in the scalar case $L(x\zeta)$,
whose $K$-types
are now well-known.
The last non-trivial
reduction
point in this scalar family
is given by $x=-\frac a2$ and 
is in some cases the minimal holomorphic representation, 
where $a$ is certain root multiplicity.
We consider
the vector-valued case
$\lambda=-k\gamma_1-a\zeta$,
 $k\ge 1$.
We find all the 
$K$-types in 
$L(-a \zeta-k\gamma_1)$
and prove that it is always multiplicity free.

For the proof we shall combine
various methods
and results. We realize the holomorphic representations
as holomorphic polynomials on the symmetric space $D=G/K$ in its Harish-Chandra realization.
The first
elementary method
is  the tensor product 
decomposition
\begin{equation}
\label{tensor-L-min}
L(-\frac a2\zeta)
\otimes
L(-\frac a2\zeta)
=\sum_{k=0}^\infty
L(-a\zeta-k\gamma_1). 
\end{equation}
This motivates 
also
our problem of
finding the $K$-types in $L(-a\zeta-k\gamma_1)$ as irreducible components in the tensor product of
copies of
minimal holomorphic representation.
Roughly speaking, 
(\ref{tensor-L-min}) 
 is the Taylor expansion
of holomorphic functions in two
variables $(z, w)\in D\times D$ 
along the diagonal $z=w$.
The $K$-type decomposition
of $L(-\frac a2\zeta)$ is \cite{Wa}
\begin{equation}
    \label{L-module}
L(-\frac a2\zeta)
=\sum_{m=0}^\infty
F(-m\gamma_1
-\frac a2\zeta).
\end{equation}
We first find
the $K$-types
in the whole
tensor product
(\ref{tensor-L-min}) 
by
decomposing
\begin{equation}
\label{tensor-F-m-l}
F(-m\gamma_1
-\frac a2\zeta)
\otimes 
F(-l\gamma_1
-\frac a2\zeta)
=\sum_{\mu}
F(\mu
-a\zeta),
\end{equation}
counting multiplicities.
This is done by  case by case computations
using Littlewood-Richardson
rule when $K$
is classical
 and some results of Stembridge
\cite{St}
when $(\mathfrak g, \mathfrak k)=(\fe_7, \fe_6 +\fso(2))$.
Thus we have produced all the $K$-types in  the tensor product (\ref{tensor-L-min}).
Next we use 
 the general
result of \cite{EJ}
on the description
of the K-types 
for general
unitary highest
weight module $L(\lambda)$ in terms of
PRV (Parthasarathy-Ranga Rao-Varadarajan) components. This
result can be effectively used
for the symmetric pair 
\[
(\fg_0, \fk_0)
=
(\fso(2, n), \fso(2)+\fso(n))
\]
of rank two to find the K-types
in $L(-a\zeta -k\gamma_1)$. 
For symmetric spaces
 $D$ of the pairs
$(\fso^*(2n), u(n))$, $n\ge 4$,
$(\fe_{6(-14)}, \fso(10)+\fso(2)),
(\fe_{7(-25)}, 
\fe_{6}+\fso(2))
$ we use
the restriction
of the decomposition
(\ref{tensor-L-min}) and their
irreducible
$K$-subspaces
as explicit holomorphic polynomials $f(z, w)$ in two variables
in $(z, w)\in D\times D$
on the subdomain
$D'\times D'$
where $D'\subset D$
corresponds
to the pair
$(\fso(2, n),
\fso(2)+\fso(n))
$.
We prove that
all $\fg$-highest weights
and $\fk$-highest weights
involved
are determined
by the subroot system corresponding
to 
$(\fg', \fk')=(\fso(2, n),
\fso(2)+\fso(n)),
$
and we find 
the $K$-types
in $L(-a\zeta -k\gamma_1)$
using the
$K'$-types
in 
$L(-a\zeta' -k\gamma_1')$
for $(\fg', \fk')$.

It is well known that for groups
$G=U(r, r+b)$, $SO^*(2n)$, or $Sp(n, \mathbb R)$, all unitary highest weight $(\fg, K)$-modules
can be obtained by dual pair correspondence from finite-dimensional representations of compact groups.
This is well studied for
$G=Sp(n, {\mathbb R})$ and the corresponding dual pair  $
(G,  M)=
(Sp(n, {\mathbb R}),  O(m))
$.
The $K$-multiplicity can be
studied by considering 
the dual pair with its see-saw  pair
$(K, L)=
(U(n) , U(m))$.
There is a reciprocity formula for the
$K$-multiplicity, 
$$\mult_{K}(F(\mu), L(\lambda))
=\mult_{L}(F(\bar\la), F(\tilde \mu))
$$
where $\mu\rightarrow \tilde \mu$ is the
dual pair correspondence for
$(K, L)$ and $\bar\la\to\la$ is the
dual pair correspondence for
$(G, M)$. This
requires nevertheless concrete computations;
in our cases
of representations
$L(-a\zeta -k\gamma_1)$
of $(\fg_0, \fk_0)=
(\mathfrak u(r, r+b),
\mathfrak u(r)+\mathfrak u(r+b)),
(\fsp(r, \mathbb R),
\mathfrak u(r))$
the computations
are rather simple and we
present thus a second proof.
One advantage of the first  approach is
that it gives a method
for finding
the highest weight vectors
in the concrete realization.

As in \cite{EW}, it is possible to use results about $K$-types to obtain some branching laws of finite-dimensional representations. We include one such very explicit result; see Corollary \ref{cor fd}.

The highest 
weight
representations
considered
here can 
all be realized
as holomorphic sections
of a subbundle
of the symmetric
tensor 
$S^k(T'(G/K))$ 
of the dual holomorphic tangent bundle
tensored with a line
bundle. 
The maximal
submodule
can be obtained
as the kernel  of a covariant differential
operator of first
degree by the main
results
of \cite{DES}.
Therefore our
result gives
the  kernel
of this differential operator.

We remark also that generally it is an impossible  problem to find explicitly
the $K$-decomposition
of  a $(\fg, K)$-module, even for highest weight representations.
Our results show that it is still possible
for certain family of representations
by combining various techniques.

The main result of this paper is the following theorem:

\begin{theo+} 
\label{thm1.1}
Consider
  the analytic continuation
  of the family  $L(\lambda)$, 
  $\lambda=-k\gamma_1 +x\zeta$,
  $x<<0$, where $k>0$ is an integer if
  $ \fg\ne \fsp(r, \mathbb C)$
  and a positive integer or half integer if
   $ \fg= \fsp(r, \mathbb C)$.
  The last reduction point   
in the family $L(\lambda)$,
$\lambda=-k\gamma_1 +x\zeta$
and corresponding
$K$-types $\mu$
in $L(\lambda)$
are given as follows:
\begin{enumerate}
    \item 
If $\fg_0= \fsu(r, r+b)$, then $\lambda=-k\gamma_1-2\zeta$,
$$
\mu= -2\zeta
+ ((0, \cdots, 
0, -\mu_2, -\mu_1))_{\fgl(r)}
+
(\nu_1, \nu_2, 0,
\cdots,0)_{\fgl(r+b)}
$$
subject to the conditions
$\mu_1+\mu_2 =\nu_1+\nu_2$,
and
$$
(\mu_1-\mu_2) + 
(\nu_1-\nu_2)
\ge
2k  
\ge 
|(\mu_1-\mu_2) - (\nu_1-\nu_2)|.
$$

\item 
If  $\fg_0= \fsp(r, \mathbb R)$ and $k\ge 1$
is an integer or a half integer, 
then $\lambda = -\zeta -k\gamma_1$
$$
\mu=
\lambda- m_1 \gamma_1-\frac
{m_2}2 (\gamma_1+\gamma_2), \quad m_1,   m_2\ge 0. 
  $$
\item  If $\fg_0$ is different from $\fsu(r,r+b)$ and $\fsp(r, \mathbb R)$, and $k\ge 1$ is an  integer, then $\lambda
=-a\zeta -k\gamma_1$,
   $$
    \mu=\lambda -m_1\gamma_1 -m_2\gamma_2 -j \alpha,
   \quad m_1\ge m_2\ge 0, \ 0\le j\le\min\{k,  \frac{m_1-m_2}2\} $$
   where $\alpha$ is the simple compact root connected 
   to the non-compact
   simple root $\gamma_1$.
    \end{enumerate}
 In all cases
 the $K$-types
 are of multiplicity one.
\end{theo+}

We shall prove Theorem \ref{thm1.1} 
through  case by case computations. To do this, we first review some necessary results about structure of the pairs $(G,K)$ in Section 2, and then some results about tensor products of representations in Section 3. After that we 
 prove Theorem \ref{thm1.1} for $(G,K)=(U(r,r+b),U(r)\times U(r+b)$ and for $(G,K)=(Sp(r,\mathbb{R},U(r))$ in Section 4, for $(G,K)=(SO_0(2,m),SO(2)\times SO(m))$ in Section 5, for $(G,K)=(SO^*(2n),U(n))$ in Section 6, for $(G,K)=(E_{6(-14)},Spin(10)\times SO(2))$ in Section 7, and for $(G,K)=(E_{7(-25)},E_6\times SO(2))$ in Section 8.

\section{Preliminaries}
We fix notation and recall some well known facts on unitary highest weight
representations; see  \cite{EHW, EJ}.

\subsection{Root spaces of
  $\fg$ and  $\fk$}
Let $D=G/K$ be an irreducible non-compact Hermitian symmetric
space and $\fg_0=\fp_0+\fk_0$ be the  Cartan decomposition
of the Lie algebra $\fg$ of $G$.
The complexification $\fg$
of $\fg_0$ has a Harish-Chandra
decomposition 
$\fg=\fp^{+}+\fk + \fp^{-}$ and $D=G/K$ 
is realized as a bounded open domain in $\fp^+$ 
via the Harish-Chandra 
realization.

We use the Jordan triple characterization of the bounded symmetric
domain $D=G/K$.
So $V=\fp^+$ is equipped with a Jordan triple product
$\{x, \bar y, z\}=D(x,  \bar y)z$ which is bi-linear
and symmetric in $x, z$ and conjugate linear in $y$ so that
when  realized as holomorphic vector fields on $D$, the Cartan decomposition
 $\fg=\fk +\fp$
takes the form
  \begin{equation}
    \label{cartan-p}
  \fp=\{\xi(v)=v-Q(z) \bar v; v\in V\}, 
  \end{equation}
  with $2Q(z) \bar v =D(z, \bar v) z$
  and
  $$
 [\xi(v), \xi(u)]= D(v, \bar u)- D(u, \bar v). 
  $$
  The underlying complex space of $\overline V=\overline{\fp^+}$ is identified with $\fp^-$.

Let $\{e_1, e_2, \dots, e_r\}$
be a Jordan frame of  $V$, where $r$ is the real rank of $G/K$.   The Peirce
decomposition of $V$ with respect to this frame is
\begin{equation}
  \label{Peirce}
V=\S_{r\ge j\ge k\ge 0}V_{jk},  \quad V_{jk}=V_{kj}:= \{v\in V; D(e_i,
e_i)v=\dfrac{\delta_{ij}+\delta_{ik}}2 v, \ i=1, \dots, r\}.
\end{equation}
The Jordan product rules
of these spaces are
\begin{equation}
\label{Peire-prod}
\{V_{jk}, \overline{V_{kl}}, V_{lm}\}
\subset 
V_{jm}    
\end{equation}
 and all other products are vanishing.
Let
\begin{equation}
\label{ab-def}
a=\dim V_{jk},\ r\ge j>k\ge 1,
\quad 
b=\dim V_{j0},
r\ge j\ge 1, 
\end{equation}
and
\begin{equation}\label{genus}
    p:=1+(r-1)+ b =\tr_{V} D(e_1, e_1).
\end{equation}
 Then the dimension $d=1+\dfrac a2 r(r-1) +rb$. The number $\frac a2$
 is the gap of the reducible points
in the analytic continuation
$L(-x\zeta)$ and $p-1$
is the limit of discrete
series; see Section 
\ref{sect-23} below.

Let $
\ft_-
\subset \fk$ 
be the abelian subspace
$$
\ft_-
=\sum_{j=1}^r\mathbb C{D(e_j, e_j)}
$$
and extend it to a Cartan subalgebra 
$\ft
=\ft_- + \ft_+ $
of $\fk$,
and thus also of $\fg$.
Let $\Delta$ be the set of roots of $\fg$
with respect to  $\ft$. Then
$\Delta= \Delta_c\cup  \Delta_n$,
union of compact and non-compact roots, 
and $\Delta^+= \Delta_c^+\cup  \Delta_n^+$
the corresponding positive roots.
Denote the
compact root space
of $\mathfrak \alpha\in \Delta_c$ by
$\mathfrak k_{\alpha}$. We choose the Harish-Chandra strongly
orthogonal roots
${\gamma_1<\gamma_2<\dots <\gamma_r}$
as the dual basis of
$\frac 12 {D(e_j, e_j)}$ and extend them
on $\ft_- + \ft_+ $
vanishing on $\ft_+ $. 
We shall also need
the highest non-compact root  $\beta$,
i.e. the highest
root of $\fp^+$;
$\beta=\gamma_r$ if
$D$ is a tube domain. 
The compact
roots are of the form
\begin{equation}
    \label{cpt-root-sp}
    \Delta_c^+|_{\ft_-}
=\{\dfrac{\gamma_j -\gamma_k}2;
j>k\ge 1\}
\cup
\{-\dfrac{\gamma_j}2;
j\ge 1\}
\end{equation}
and 
$
\fp^{+}_{jk}=V_{jk}$, 
$\fp^{+}_{j}=V_{j0}
\subset \fp^+=V$ are  {\it positive  non-compact} root spaces
with roots
$\dfrac{\gamma_j +\gamma_k}2, \dfrac{\gamma_k}2  
$. Let 
\begin{equation}
\label{pos-cpt-rts}
\mathfrak k^+
=
\sum_{j>k\ge 1}
\fk_{jk}
+\sum_{j\ge 1}\fk_{0j}
\end{equation}
be
the  sum
compact positive root spaces of 
with respect to
$\ft_-$.
 They
 are described in \cite[Lemma 1.5, Theorem
1.7(1)]{Up-AJM};
note that
the co-root of $\gamma_j$
in 
\cite{Up-AJM}, 
 is numbered
as $C_{r+1-j}$.

\begin{lemm+}
  \label{root-sps}
  The compact positive
  and zero root spaces 
with respect
to $\mathfrak t_-$
are given by
$$
\fk_{jk}
=\text{Span}\{D(u, \bar v);  \quad u\in V_{jl},  v\in V_{lk}, 0\le l\le r\}, \ 
1\le k< j \le r; 
$$
$$
\fk_{0j}
=\text{Span}\{D(u, \bar v); 
\quad u\in V_{0l}, v\in V_{lj}, 0\le l\le r; \},\ 
1\le j \le r,
$$
$$
\mathfrak t\oplus \sum_{\alpha,\,
\alpha|_{\mathfrak t_-}
=0}^{\oplus}
\mathfrak k_{\alpha}
=
\sum_{j=0}^r 
\mathfrak k_{jj}, \quad 
\mathfrak k_{jj}= 
\text{Span}\{D(u, \bar v); 
\quad u\in V_{jk}, v\in V_{kj}, 0\le k\le r \}.
$$
 \end{lemm+}

Observe that the second sum above
is not a direct sum.

%and $D(e_r, \bar v)=0$ if $v\in V_{jk}$ for $r>j\ge k$.
%% \begin{proof} 
%%   for a fixed $j$, then 
%%   $$
%%   [ D(e_r, \bar e_r), D(e_r, \bar v)]
%%    =   D(D(e_r, e_r)e_r,  \bar v)-
%%    D(e_r, \overline{D(e_r, e_r) v}) 
%%    = 2 D(e_r,  \bar v)-
%%       D(e_r, \overline{v}) =D(e_r,  \bar v) 
%%    $$
%%    and 
%% $$      [ D(e_k, \bar e_k), D(e_r, \bar v)]
%%    =   D(D(e_k, e_k)e_r,  \bar v)-
%%    D(e_r, \overline{D(e_k, e_k) v}) 
%%    = -  \delta_{jk}D(e_r, \overline{v}) 
%% $$
%%  for $k\le 1$. This proves the first claim. The second 
%%  one $D(e_r, \bar v)=0$ for $v\in V_{jk}$, $r>j\ge k$, 
%%  is a general fact for Jordan triple \cite{Loos-bsd}. 
%% \end{proof}

 For any root $\alpha\in
 \Delta^+$ let $\alpha^{\vee}=\dfrac{2\alpha}{\(\alpha,
  \alpha\)}$
be the corresponding coroot. Let $\rho$ be the half sum of positive
roots in 
$\Delta^+$.
Then
 $$
 \(\rho, {\beta}^{\vee}\) = p-1=1+a(r-1)+b.
  $$

\subsection{The space of polynomials on $\fp^+$}

We recall first the 
Hua-Kostant-Schmid decomposition \cite{Schmid}.
The space $\mathcal P(\mathfrak p^+)$
of holomorphic polynomials on $\fp^+$ 
has the
decomposition
\begin{equation}
\label{poly-deco}
\mathcal P(\fp^+)
=\sum_{m_1\ge m_2\ge \dots \ge m_r\ge 0}
F(-\m)
\end{equation}
under $\mathfrak k$, 
where each space
$F(-\m) 
$ has 
highest weight
\[
-\m= 
{-m_1\gamma_1- m_2\gamma_2-\dots - m_r\gamma_r}
\]
with $m_1 \ge \dots\ge m_r\ge 0$ being integers.

\subsection{Unitary
highest weight representations}
\label{sect-23}

The classification of the unitary highest weight
representations of $G$
was done in \cite{EHW, Jak}; see also \cite{PPST,PPSST}.
We fix the central character $\zeta$ of 
$\mathfrak k
$ by
$$
\zeta (X) =\frac{1}{p}\tr_{\fp^+} \ad X, \quad X\in  \mathfrak k.
$$
Then $\zeta(D(e_1, e_1))=\zeta(\beta^\vee)=1$.
Let $\lambda \in\ft^\ast$ 
be the highest weight of a finite-dimensional representation $F(\lambda)$ 
of $K$, and denote by $\tau_\lambda$ the action of $K$
and by $x(\lambda)$
the central character of $F(\lambda)$ (shifted by $\rho(\beta^\vee)$)
$$
x(\lambda )
=(\lambda +\rho)(\beta^\vee).$$
When $x(\lambda) <0$,
there is a holomorphic
discrete series
$(L(\lambda), \pi_{\lambda})$ of $G$
with highest weight $\lambda$.
Writing $\lambda $  as 
\[
\lambda =
(\lambda-x(\lambda)\zeta)
+ x(\lambda)\zeta=
\lambda_0+x(\lambda)\zeta,
\] 
$x(\lambda)=0$
corresponds to  the limit
of the discrete series
and $x(\lambda)$ is the parameter used in
\cite{EHW} for the
classification of unitary highest weight representations.

The  
$(\mathfrak g, K)$
module $L(\lambda)$
can be realized
as a subspace of 
holomorphic
polynomials on $\fp^+$
taking values in $F(\lambda)$
and we shall use this realization.
We shall need 
the precise
formulas
for actions  of 
elements in $ \mathfrak k+\fp^+$. 
They are given by \cite{JV}
\begin{equation}
\label{k-p-act}
(\pi_\lambda (X) p)(z) = 
-(\partial_X p)(z),\ \  
(\pi_\lambda(Y) p)(z) = 
\tau_\lambda(Y) p(z)
- (\partial_{[Y,  z]} p)(z), \ 
z\in \fp^+
\end{equation}
where 
$X\in \fp^+, Y\in \fk$
and $\partial_X$
is the differentiation
along $X$ on functions.

When $\lambda= z\zeta$ 
is of scalar  type
the analytic continuation of $L(\lambda)$  was found 
earlier in \cite{Wa} and the 
unitarity points are 
the set  
$$
z\in  (-\infty, -\dfrac a2(r-1))\cup\{-\dfrac a2(r-j),\ j=1, \dots, r\}
$$
At the last non-trivial reduction point we have
$L(-\dfrac a2\zeta)$
is a Ladder representation,
\begin{equation}\label{last-W-pt}
L(-\dfrac a2\zeta)
=\sum_{m\ge 0}F({-\dfrac a2\zeta -m\gamma_1}).
\end{equation}
When $\fg_0=\fsp(r, \mathbb R)$
the metaplectic representation
of $Mp(r, \mathbb R)$
is a direct
sum $
L(-\frac 12\zeta)
+
L(-\frac 12\zeta -\epsilon_r)
$ as a representation 
of $\fg_0=\fsp(r, 
\mathbb R)$,
where $-\epsilon_r$
corresponds to  the dual representation
of $U(r)$ on $\mathbb C^r$.
$
L(-\dfrac 12\zeta-\epsilon_r)
$
is also a ladder
representation
$$
L(-\dfrac 12\zeta
-\epsilon_r)
=\sum_{m\ge 0}F({-\dfrac 12\zeta -m\gamma_1}
-\epsilon_r)
=\sum_{m\ge 0}F(-\dfrac 12\zeta -(2m+1)\epsilon_r).
$$

\section{Tensor products
of representations}

The method
of tensor products
has been a powerful
method in constructing
unitary highest weight
representations
of $\fg$. Here we
shall also use this method to find
$K$-types in 
some of these
representations.

\subsection{Tensor product 
$L(-\dfrac a2\zeta)   \otimes 
  L(-\dfrac a2\zeta)$ }

Tensor products of
holomorphic discrete
series representations have been studied
in \cite{Re}. Here we shall
need the tensor product
of  representations at
the last non-trivial
 Wallach point.
We prove a slightly
 more general result,
which can also be
deduced from
the argument in 
\cite{JV}.
We  apply it then to our
case. 

\begin{lemma}
Consider 
the tensor product
$L(\mu)
\otimes 
L(x\zeta )
$
of 
 an arbitrary
unitary 
 highest weight
representation
 $L(\mu)$
and
a unitary 
scalar 
highest weight
representation 
 $L(x\zeta)$.
 Let 
$$
L(\mu)
\otimes 
L(x\zeta )
=\sum_{\nu} c_{\nu}L(\nu)
$$
be its decomposition into 
irreducible 
representations, 
where $c_\nu$
is the multiplicity.
Then any $\fg$-highest
weight vector $F$ in 
$c_{\nu}L(\nu)$
is of the form $F(z, w)
=f(z-w)$
where $f\in 
L(\mu)$ is a $\fk$-highest
weight vector in $L(\mu)$.
\end{lemma} 
\begin{proof}  
The space $L(\mu)$
in our setup 
is realized as
a subspace of holomorphic
polynomials
on $\mathfrak p^+$ taking values
in $F(\mu)$
with highest weight $\mu$.
The action of $\mathfrak g$ on general
$L(\mu)$ is given 
in \cite[(2.12), p.~35]{JV}. 
Suppose 
$F$
in the tensor product is $\fg$-highest
weight vector.
Then 
$(\pi_{-\frac 12\zeta
-\epsilon_r}
\otimes 
\pi_{-\frac 12\zeta})(X)F
=0
$ for $X\in \fp^++\fn_{\fk}$.
Using (\ref{k-p-act})
this implies that
$$
(\partial_X^{z}
+\partial_X^{w})F(z, w)=0, \ X\in \fp^+.
$$
It follows
that $F$ 
must of the form
$F(z, w)=f(z-w)$
for a
polynomial $f$ on $\fp^+$
with $f(z)=F(z-0)$
being an element
in $L(-\frac 12\zeta
-\epsilon_r)$.
If it is further
a highest weight 
vector 
with respect to
$\fn_\fk$, then for   
any $X\in \fn_{\fk}$,
$$
((\pi_{
\tau(\mu)
}
\otimes \pi_{\tau(-x\zeta)})
(X)F)(z, w)
$$
$$=
d\tau(\mu)
(X)f(z-w)
+d\tau(x\zeta)
(X)f(z-w)
-\delta([X, z])F(z, w)
-\delta([X, w])F(z, w)
$$
$$
=d\tau(\mu)
(X)f(z-w)
-\delta([X, z-w])f(z-w)=0,$$
since $\zeta$ is a
central character
and $\zeta(X)=0$.
Restricted to $w=0$
this is
$$
d\tau(X)f(z)
-\delta([X, z])f(z)=0.
$$
which means
 $f$ is a $\fk$-highest
 weight vector in 
$L(\mu)$, and
$F(z, w)
=f(z-w)$ as claimed.
\end{proof}

\begin{theorem} 
\label{tensor-min}
\begin{enumerate}
    \item
(\cite[Theorem 4.4]{PZ}.)
Let $L(-\frac a2\zeta)  $
be realized as above, 
as a subspace of
holomorphic polynomials on $D=G/K$.
We have
decomposition of the tensor product
\begin{equation}
  \label{ten-prod}
  L(-\frac a2\zeta) 
  \otimes 
  L(-\frac a2\zeta)
\simeq  \sum_{k=0}^\infty
L(-a\zeta -k\gamma_1)
\end{equation}
where 
$L(-a\zeta -k\gamma_1)
  $
is
realized as a subspace of polynomials $F(z, w)$
of $(z, w)\in D\times D$
  in   $L(-\frac a2\zeta)   \otimes 
  L(-\frac a2\zeta)$
  with degree of vanishing
   along the diagonal $z=w$
  in $D\times D$ being exactly $k$.
Up to nonzero constants, the
intertwining operator 
from 
$L(-\frac a2\zeta)   \otimes 
  L(-\frac a2\zeta)$
to  $ L(-a\zeta -k\gamma_1)$
is uniquely given by
\begin{equation}
\label{bi-lin-int}    J_{k} p(z)
=\sum_{i=0}^k
(-1)^i \binom{k}{i}
\frac{1}{
(\frac a2)_i
(\frac a2)_{k-i}
}(
\partial_z^i \otimes \partial_w^{k-i}) p(z, z),\quad 
p\in L(-\frac a2\zeta)\otimes
L(-\frac a2\zeta), 
\end{equation}
where $(\nu)_j=\nu(\nu+1)
\cdots(\nu+j-1)$ is the Pochammer symbol.
\item Let 
$\fg_0=\fsp(r, \mathbb R)$
with $\fp^+$ 
realized as 
$\fp^+=\{z=(z_{ij})\in M_{r, r}; z=z^t\}$
of symmetric matrices. We have
the following tensor product
decomposition
\begin{equation}
\label{tensor-product-var}
L(-\frac 12\zeta-\epsilon_r)
  \otimes 
L(-  
\frac12\zeta)  
  \simeq   \sum_{k=0}^\infty 
  L(-\zeta-(2k+1)\epsilon_r).
\end{equation}
The subspace $ L(-\zeta -(2k+1)\epsilon_r)$
is realized as
the subspace of 
holomorphic functions
$f(z, w)$ in 
$L(-\frac 12\zeta-\epsilon_r)
  \otimes 
L(-  
\frac12\zeta)  $
which are vanishing
of exactly degree $2k$ along
the diagonal $z=w$,
and the corresponding
intertwining operator
from 
$L(-\frac 12\zeta)   \otimes 
  L(-\frac 12\zeta-\epsilon_r)$
to  
$ L(-\zeta -(2k+1)\epsilon_r)$
is given by
\begin{equation}
\label{bi-lin-int-v}   
J_{k} p(z)
=\sum_{i=0}^k
C(k, i)
(\partial_z^i \otimes \partial_w^{k-i}) p(z, z)
\end{equation}
for constants
$C(k, i)$
with the leading
term $C(k, k)\ne 0$.

\end{enumerate}
\end{theorem}

For the classical
Lie groups
$U(r, r+b)$, $ 
Sp(r, \mathbb R)$ and $
SO^*(2n)$, the
first statement  can also
be obtained
using dual pair
correspondences
\cite{KV, NOT}.
The second part
can be derived 
from the techniques
in \cite{JV, PZ}
and we sketch a proof.
\begin{proof}
(2) Let $\{e_1, \cdots, e_r\}$
be the standard
basis vectors in 
$\mathbb C^r$
of weights $-\epsilon_1, \cdots, -\epsilon_r$, for 
$\mathfrak k=
\mathfrak{gl}(r, \mathbb C)$ acting on $\mathbb C^r$ by the
dual action of the
defining representation.

It is easy
to check that
$F(z, w)=(z_{rr}-w_{rr})^k
e_{r}$ 
is in the tensor product and is
of highest
weight $
\lambda
=-\zeta -k\gamma_1-\epsilon_{r}$. Thus
$L(\lambda) $ appears
in the tensor product 
(\ref{tensor-product-var}).
The rest follows by the same
arguments
as in \cite{JV, PZ}
\end{proof}

\subsection{
$K$-types in the
tensor products
  $(  \mathcal L(-\frac a2\zeta)
\otimes
  \mathcal L(-\frac a2\zeta),
  \mathfrak g)$
and the corresponding
$K'$-types in
  $(  \mathcal L'(-\frac a2\zeta)
\otimes
  \mathcal L'(-\frac a2\zeta),
  {\mathfrak g'})$
  for a tube subdomain $D'=D\cap V_2$  }
  \label{Sub-sym-pair}

Let $D$ be the symmetric 
domain in $V$. Let $1<s\le r$,  fix the 
$s$ orthogonal roots $\gamma_1<\cdots < \gamma_s$
and the tripotents $e_1,
\cdots, e_s$. Consider
the Peirce decomposition  $V=
V_2 + V_1 + V_0$
with respect to $e'=e_1+\cdots +e_s$,
\begin{equation}
\label{V-2-1-0}
V_{2}=\sum_{1\le j\le k\le s} V_{jk},
\,
V_{1}=\sum_{1\le j\le s <k} V_{jk} +
\sum_{1\le j\le s } V_{j0}, 
\,
V_{0}=\sum_{s< j\le k \le r} V_{jk} + 
\sum_{s< j\le r} V_{j0};
\end{equation}
where $V_{jk}$ are defined in (\ref{Peirce}).
Let
$$
P_2: V\to V_2
$$
be the orthogonal projection.

Let $D'=V_2\cap D$ be the corresponding
 tube subdomain in $D$ 
 of rank $s$ 
 and $ \mathfrak g', \mathfrak k'$
etc be the corresponding complexified
 Lie algebras.
We shall need the following simple observation
on the extension of the highest weight vectors.
We shall need somewhat more
precise description
of the root vectors
in
Lemma 
  \ref{root-sps}
 that are
also in $\mathfrak k'$.

\begin{lemm+}
  \label{root-sps}
Fix $1<s\le r$ as above.  The compact positive root spaces 
of $\mathfrak k'$
    are given by
$$
\fk_{jk}=\fk_{jk}'
=\text{Span}\{D(u, \bar v);  \quad u\in V_{jl},  v\in V_{lk}, k\le l\le j\}, \ 
1\le k\le j \le s; 
$$
$$
\fk_{0j}=\fk_{0j}'
=\text{Span}\{D(u, \bar v); 
\quad u\in V_{0l}, v\in V_{lj}, 1\le l\le j \},\ 
1\le j \le s; 
$$
$$
\fn_{c}\cap \sum_{j=1}^s
\mathfrak k_{jj}
=\fn_{c}'\cap \sum_{j=1}^s\mathfrak k_{jj}'
\subset \sum_{j=1}^s\sum_{k=0}^s
\text{Span}\{D(u, \bar v); 
\quad u\in V_{jk}, v\in V_{kj}, 0\le k\le s \}. 
$$
 \end{lemm+}
\begin{proof}
Our claim can be checked
by a case by case computation. For 
classical domains this
follows directly
from the formulas
for the root system.
If  $(\mathfrak g, 
\mathfrak k)=(\mathfrak
e_6, \mathfrak{so}(10, \mathbb C) +\mathbb C)$, respectively  $(\mathfrak e_7,
\mathfrak e_6 +\mathbb C), 
$ then 
 $(\mathfrak g', 
\mathfrak k')=
(
\mathfrak{so}(10, \mathbb C),
\mathfrak{so}(8, \mathbb C) +\mathbb C)$, respectively 
$(\mathfrak{so}(12, \mathbb C), 
\mathfrak{so}(10, \mathbb C) +\mathbb C)$.
We shall specify these
subspaces and the above inclusions  concretely
in Sections \ref{e-6-case}
and \ref{e-7-case}.
\end{proof}

\begin{lemm+}
\label{reduct-subdom}
Fix $1< s\le r$ and let
$({\mathfrak g'},
{\mathfrak k'})
$ be as above.
\begin{enumerate}
\item
Let $F'(\mu), F'(\nu) \subset \mathcal P(V_2)$ be two
  components in the decomposition (\ref{poly-deco}) 
  for $V_2$
  with $\mu=-(m_1\gamma_1 +\cdots +m_s\gamma_s)$,
   $\nu=-(n_1\gamma_1 +\cdots +n_s\gamma_s)$. 
Consider the
irreducible 
decomposition of
  $
  F'(\mu) \otimes
    F'(\nu)
  $ and let
  $F'(\omega)$
  be an irreducible
  component.  
  Assume that $\omega$ is also dominant as a
  linear form
  on     $\mathfrak t$
  via its projection on
  $\mathfrak t'$.
  Then the  component $ F'(\omega)$ generates
  an irreducible component $ F(\omega)$ in
  the tensor product
  decomposition of $ F(\mu) \otimes
  F(\nu)$. Moreover, the corresponding highest
  weight vector of  $ F(\omega)$ is
  given by
  \begin{equation}
  \label{g-f}
 g(z, w)=f(P_2z, P_2w), \ \ (z, w)\in D\times D
  \end{equation}
  where $f$ is the  highest weight vector in
  $ F'(\omega)$
  and $P_2$ is the orthogonal projection $V\to V_2$.
\item Consider the 
 decomposition
(\ref{ten-prod}) for $({\mathfrak g'},
K')$. Let $f\in 
( L'(-\frac a2\zeta)   \otimes 
  L'(-\frac a2\zeta))_l
$ be a $K'$-highest
weight vector of weight $\omega$ considered
as a polynomial in $D'\times D'$. Assume as above that $\omega$ is also dominant
on $\mathfrak t$. 
Then the polynomial $g(z, w)$
defined in  (\ref{g-f})
is also a $K$-highest weight vector in 
$( L(-\frac a2\zeta)   \otimes 
  L(-\frac a2\zeta))_l
  = L(-a\zeta 
  -l\gamma_1)$
\end{enumerate}

\end{lemm+}
\begin{proof} (1)
First note that 
$F'(\mu)\subset F(\mu)$,
where any
element in the space
$F'(\mu)$ 
is viewed as polynomial
on $V$ via the projection $P_2: V\to V_2$,
since the  $K'$-highest 
weight vector in 
$F'(\mu)$  is precisely
the highest weight vector
in $F(\mu)$; see e.g.  \cite{FK, Schmid, Up-AJM}. 
In particular 
the element $g\in F(\mu)\otimes
F(\nu)$. 
Now $\omega$ is
$K$-dominant and thus there is an irreducible
representation of $
\mathfrak k$ with highest
weight $\omega$,
and we prove that $g$
is a highest weight vector of 
weight $\omega$
by computing
 the action $Xg$ of
positive root
vector $X\in \mathfrak k_{\alpha}$, 
$\alpha\in \Delta_c^+$.
Clearly, $g$ is of weight $\omega$ 
by our realization of $\mathfrak k'$. 
By definition we have
the tensor product action of $X
\in  
\mathfrak k$ on $f\in  F'(\mu)
\otimes
 F'(\nu)
$
as polynomials of two variables
on $V_2$ is
  \begin{equation}
    \label{hwv-f}
    (X f)(z, w)=\partial_{Xz}^{(1)} f(z, w)+
    \partial_{Xw} ^{(2)}f(z, w),\quad z, w\in V_2.
      \end{equation}
Similarly, the action of $X\in \mathfrak k$ 
on  $g\in \mathcal F(\mu)
\otimes \mathcal F(\nu)$ is
$$
(X g)(z, w)=\partial_{P_2Xz}^{(1)} f(P_2z, P_2w)+
\partial_{P_2Xw}^{(2)} f(P_2z, P_2w).
$$
We prove that $Xg=0$ for all positive root vectors $X$. Notice that 
\[
\partial^{(1)}_v g(z, w)=
\partial^{(1)}_v f(z, w)=
\partial^{(2)}_vg(z, w)=
\partial^{(2)}_vf(z, w)=0
\] 
for any $v\in V_1+V_0$ by definition.

Let
$X\in \fk_{jk},
r\ge  j> k\ge 1$.
If $j\le s$ then $\fk_{jk}=
\fk_{jk}',
$ 
and 
is spanned by the root vectors $D(v_{jl}, \bar{v_{lk}})$, 
for $1\le l\le s$
by Lemma  \ref{root-sps}. 
In the above formula for $Xg$,
we have
$P_2Xz = P_2 X(P_2z)$, for any  $z\in V=V_2+V_1+V_0$,
and $P_x Xz=Xz$ for $z\in V_2$ since
$X\in \fk_{jk}, s\ge j\ge k\ge 1$
by using  (\ref{V-2-1-0})
 and 
Peirce product rule 
(\ref{Peire-prod}).
Thus
\begin{equation*}
\begin{split}
(X g)(z, w)
&=\partial_{P_2XP_2z}^{(1)} f(P_2z, P_2w)+
\partial_{P_2XP_2w}^{(2)} f(P_2z, P_2w)\\
&=\partial_{Xz'}^{(1)} f(z', w')+
\partial_{Xw'}^{(2)} f(z', w')=0, \,\,  
\end{split}
\end{equation*}
where $z'=P_2z, 
w'=P_2w$,  
since $f$ is a highest
weight vector.
If
$j>s$ then 
$P_2Xz=0,
$
again by the Peirce rule, and we also get  $Xg=0$.

Let
$X\in \fk_{0j}, 1\le j\le r$.
If $j\le s$, 
then
$X\in \fk_{0j}'$
is a linear combination of 
\[
D(u, v),\ u\in V_{0k}, v\in V_{kj},\quad 1\le k\le j\le s,
\]
again by Lemma \ref{root-sps}.
We have $Xg=0$
by exactly the same argument as above.
If $j>s$, we
can take $X=D(u, \bar v) $ for $u\in V_{0k}$, and $Dz=0$ unless $z\in V_{jl}$ for some $l$, but then
$P_2 Xz=P_2D(u, v)z=0$ and $Xg=0$.

Let $X$ be a root vector
in the space $\mathfrak k_{jj}$,
$1\le j\le r$.
If $j\le s$, we have
$\mathfrak k_{jj}
=\mathfrak k_{jj}'$ and again $Xg=0$.
If $j>s$ we have $P_2 Xz=0$ by the Peirce relation and then $Xg=0$.

(2) We have that 
 the $K'$-module 
 \[
 L'(-\frac a2\zeta)
 =\sum_{m=0}^\infty 
 F'(-\frac a2\zeta)
 \otimes
 F'(-m\gamma_1),
 \]
viewed as
a space of polynomials
on $V$ via the projection 
$P_2: V\to V_2$,
is  in 
$L(-\frac a2\zeta )$,
by (1), where $\zeta$
on $\mathfrak k$
is also viewed
as a central character for
$\mathfrak k'$ with the central 
element $Z\in  \mathfrak k$
being viewed as in 
$\mathfrak k'$
(acting on $V_2$ by $i$).
Thus the 
polynomial $g$
above is in 
$
L(-\frac a2\zeta)
\otimes L(-\frac a2
\zeta)
$. 
Since $f$ vanishes
of degree exactly $l$
along the diagonal $z=w$
on $D'\times D'$, the
function $g$
has the same vanishing property and thus $g\in
(L(-\frac a2\zeta)
\otimes L(-\frac a2(j-1)
\zeta))_l$. Finally, the fact that 
$g$ is a highest weight
vector follows from the same
computation as in (1) above, now $g$ being
in the sum of tensor products instead of
one tensor product.
\end{proof}

\subsection{Littlewood-Richardson rule}
We shall
need the Littlewood-Richardson rule
for the decomposition of the tensor product
of {polynomial} representations of $\mathfrak{gl}(n, \mathbb C)$; see e.g. \cite[Chapter 9]{GW}. 
Denote by $F(\alpha)$
the representation of 
$\mathfrak{gl}(n, \mathbb C)$
with highest weight
$\alpha=(\alpha_1,
\cdots, \alpha_n)$,
$\alpha_1\ge \cdots\ge\alpha_n\ge 0$,
corresponding to a Young diagram with $\alpha_1$ boxes in the first row, $\alpha_2$ boxes in the second row, etc.
The multiplicity $m(F(\gamma),
F(\alpha)\otimes F(\beta))$,
with $\alpha\preceq\gamma$, i.e.,
$\alpha_j\le \gamma_j\ \forall j$,
is given by the Littlewood-Richardson coefficient $LR_{\alpha, \beta}^\gamma$
defined as the number of 
semistandard skew
tableaux of shape $\gamma-\alpha$
whose weights are given 
by the fixed $\beta=(\beta_1, \beta_2, \cdots)$. More
precisely, the 
skew tableau $\gamma-\alpha$
is obtained from the
Young diagram of $\gamma$
with the diagram of
$\alpha$ being removed, and 
a semistandard skew
tableau of weight
$\beta=(\beta_1, \beta_2, \cdots)$ is a filling of 
 $\beta_1$ boxes of $1$'s
 $\beta_2$ boxes of $2$'s etc,
 inserted in the
tableau $\gamma-\alpha$, 
such that
\begin{itemize}
\label{LR-R}
\item  each row
is weakly increasing
from left to the right;
\item 
each column is strictly increasing
from top to the bottom;
\item
the word obtained from top to bottom,
right to left, is a Yamanouchi word, namely at each fixed position $j$ in the word, 
the number of any $i$ appearing
before the position 
$j$ 
is not smaller than
the number of $i+1$ appearing 
before the position $j$.
\end{itemize}

\begin{lemma}
\label{lem LR2}
We have the following 
for representations of $U(n)$, $n\ge 6$:
\begin{enumerate}     
\item The following formula holds,
\begin{equation*}
\begin{split}
&\quad F(\nu_1,\nu_1,\nu_3,\nu_3,0,\dots,0)\otimes F(k,k,0,\dots,0)
\\ 
&=
\bigoplus F(\nu_1+l_1,\nu_1+l_2,\nu_3+l_3,\nu_3+l_4,l_5,l_6,0,\dots,0),
\end{split}
\end{equation*}
with each summand having multiplicity 1, and with the sum running over integers $l_1,l_2,l_3,l_4,l_5,l_6$ such that 
\begin{eqnarray}
\label{cond 1 LR2}
&& \qquad 0\leq l_4\leq l_3\leq \nu_1-\nu_3;\qquad 0\leq l_6\leq l_5\leq\nu_3;\\
&&\nonumber 
l_1+l_4+l_6\leq k;\qquad l_2+l_4+l_5\leq k\le l_1+l_4+l_5;\\
&&\nonumber 
l_1+l_2+l_3+l_4+l_5+l_6=2k.
\end{eqnarray}
\item The module $F=F(a_1,a_2,a_3,a_4,a_5,a_6,0,\dots,0)$, $a_1\geq \dots\geq a_6\geq 0$, appears in
\begin{equation}
\label{sum-v-fixed-k}
\left(\bigoplus_{\nu_1\geq\nu_3\geq 0} F(\nu_1,\nu_1,\nu_3,\nu_3,0,\dots,0)\right)\otimes F(k,k,0,\dots,0)
   \end{equation}
with nonzero multiplicity if and only if the following conditions hold. Denote by $t_1=a_1-a_2$, $t_2=a_3-a_4$, $t_3=a_5-a_6$ the tails of $F$. Then $t_1+t_2+t_3=2d$ must be even, i.e., $d=\frac 12 (a_1+a_3+a_5-a_2-a_4-a_6) $ must be an integer; clearly, $d\leq k$. 
Furthermore, we must have
\begin{eqnarray}
\label{cond a LR2}
&&t_1\leq d;\qquad t_3\leq d;\qquad t_1+t_3\geq d; \qquad a_6\leq k-d;\\
&& \nonumber k-d-a_6\leq (a_2-a_3)+(a_4-a_5).
\end{eqnarray}
Suppose these conditions are satisfied, and denote $u=a_2-a_3$, $v=a_4-a_5$, $w=k-d-a_6$. Then the multiplicity of $F$ equals the number of paris
$(r,s)$ satisfying
\begin{eqnarray}\label{rs}
&&0\leq r\leq u;\\
&&\nonumber 0\leq s\leq v;\\
&&\nonumber r+s=w.
\end{eqnarray}
This number can be described explicitly as 
\[
\min(w,u,v,u+v-w)+1
=\min(k-d-a_6, a_2-a_3, a_4-a_5, a_2+a_4+a_6-(a_3+a_5)-(k-d)) +1
\]
\item When $\nu_3=0$ and $n\ge 4$ we have
\begin{eqnarray*}
F(\nu_1,\nu_1,0,\dots,0)\otimes F(k,k,0,\dots,0)=
\bigoplus F(\nu_1+l_1,\nu_1+l_2,l_3,l_4,0,\dots,0),
\end{eqnarray*}
with each summand having multiplicity 1, and with the sum running over $l_1,l_2,l_3,l_4$ such that 
\begin{eqnarray}\label{cond l LR1}
&&l_3\geq l_4\geq 0;\qquad l_3\leq\min(\nu_1,k);\\
&&\nonumber l_2=k-l_3;\qquad l_1=k-l_4.
\end{eqnarray}
 The multiplicity of $F=F(a_1,a_2,a_3,a_4,0,\dots,0)$, $a_1\geq a_2\geq a_3\geq a_4\geq 0$, in
\[
\left(\bigoplus_{\nu_1\ge 0} F(\nu_1,\nu_1,0,\dots,0)
\right)
\otimes 
F(k,k,0,\dots,0)
\]
and
respectively
in 
\[
\left(\bigoplus_{\nu_1\ge 0} F(\nu_1,\nu_1,0,\dots,0)
\right)
\otimes 
\left(\bigoplus_{k\ge 0} 
F(k,k,0,\dots,0)
\right)
\]
is 
\begin{equation}
\label{cond a LR1}
\begin{cases} 1, &    a_1-a_2=a_3-a_4,\ a_3\leq k\leq a_2,\\
0, &\text{otherwise}
\end{cases}
\end{equation}
and respectively
\begin{equation}
\label{cond a LR1-var}
\begin{cases} a_2-a_3+1, &    a_1-a_2=a_3-a_4, \\
0, & \text{otherwise}.
\end{cases}
\end{equation}
\end{enumerate}
\end{lemma}
Notice that
the obvious condition $0\le l_2\le l_1$ is
a consequence of
the inequalities in (\ref{cond 1 LR2}).

It is possible to extend the above statements to include also the cases $n<6$. We leave the details to the reader.

\begin{proof} (1)
As usual, to satisfy the Yamanouchi and the semistandard conditions, the first line of the skew tableau must consist of $l_1$ ones, the second line must consist of $l_2$ twos, the sixth line must consist of $l_6$ twos, and the fifth line must consist of $l_6$ ones followed by $l_5-l_6$ twos. Furthermore, the fourth line must consist of $l_4$ twos, and then the third line will consist of $l_3'=k-l_1-l_6$ ones and $l_3''=k-l_2-l_4-l_5$ twos.

Since $l_3'\geq l_4$ (we must have ones above twos), it follows that $l_1+l_4+l_6\leq k$. Since the total number of twos is $k$, we see that $l_2+l_4+l_5\leq k$.
Finally, the nontrivial place to check the Yamanouchi condition is the first occurrence of two from the left in the third row; we must have $l_3''+l_2\leq l_1$, which is equivalent to $l_1+l_4+l_5\geq k$. The other conditions are obviously necessary. 

It also follows that if all the conditions are met, then there is a unique way to place ones and twos, so the corresponding module appears in the tensor product with multiplicity one. 

(2) It is clear that the tails $t_1=a_1-a_2, t_2=a_3-a_4, t_3=a_5-a_6$ must belong to the skew diagram in the tensor product decomposition of
(\ref{sum-v-fixed-k}), and that in total they must contain an equal number of ones and twos; denote this number by $d$, so that $t_1+t_2+t_3=2d$. 

We know that $t_1$ must consist of ones and $t_3$ must consist of twos; let $t_2$ consist of $t_2'=d-t_1$ ones and $t_2''=d-t_3$ twos. Since the blocks \ $\tiny{\begin{ytableau} 1 \\ 2\end{ytableau}}$\  enter into the Yamanouchi condition trivially, the Yamanouchi condition must hold for the tails, which implies $t_2''\leq t_1$, or equivalently, $t_1+t_3\geq d$.

It is clear that the last two rows of the skew diagram must consist of $a_6$ blocks \ $\tiny{\begin{ytableau} 1 \\ 2\end{ytableau}}$. The rest of the skew diagram consist of $k-d-a_6=w$ blocks \ $\tiny{\begin{ytableau} 1 \\ 2\end{ytableau}}$, and they can be placed into the first two lines, or the third and the fourth lines. The number $r$ 
of the blocks in the first two lines is at most $a_2-a_3=u$, and the number  $s$ of the blocks in the third and fourth lines is
at most $a_4-a_5=v$, which proves the last condition of 
(\ref{cond a LR2}). We 
also see that the multiplicity is indeed the number of pairs $(r,s)$ satisfying \eqref{rs}.

To see that this number is indeed equal to $\min(w,u,v,u+v-w)+1$, one can proceed geometrically: draw the rectangle $[0,u]\times [0,v]$ in the coordinate plane, and find the number of integer points in this rectangle that are also on the line $r+s=w$. Clearly, there are no such points unless $0\leq w\leq u+v$. Now one treats separately the following cases:
\begin{enumerate}
\item[(a)] $0\leq w\leq u,v$. The number of points is $w+1=\min(w,u,v,u+v-w)+1$.

\item[(b)] $v<w\leq u$. The number of points is $v+1=\min(w,u,v,u+v-w)+1$.

\item[(c)] $u<w\leq v$. The number of points is $u+1=\min(w,u,v,u+v-w)+1.$

\item[(d)] $u,v<w\leq u+v$. The number of points is $u+v-w+1=\min(w,u,v,u+v-w)+1.$

\end{enumerate}

(3) This is a consequence of (1) and (2).
\end{proof}

We are going to apply the above cases of the Littlewood-Richardson rule to describe the $K$-types of 
$L(-\frac{a}{2}\zeta)\otimes L(-\frac{a}{2}\zeta)$. For this we have to pass 
from rational $K$-types to polynomial $K$-types, i.e., the $K$-types $F(\mu_1,\dots,\mu_n)$ with $\mu_1\geq\dots\geq\mu_n\geq 0$ as polynomial 
representations
of $U(n)$ and $GL(n, \mathbb C)$. To use the above results, we need a large number of zeros at the end of the highest weights; to ensure this, in some of the cases we will use the following easy lemma.

\begin{lemm+} 
\label{dual tensor}
Let $F(\alpha)$, $F(\beta)$ and $F(\gamma)$ be finite-dimensional representations of $U(n)$. Then the multiplicity of $F(\gamma)$ in $F(\alpha)\otimes F(\beta)$ is the same as the multiplicity of $F(\gamma)^*$ in $F(\alpha)^*\otimes F(\beta)^*$.
\end{lemm+}

\begin{proof}
By complete reducibility of finite-dimensional 
$U(n)$-modules, the multiplicity of $F(\gamma)$ in $F(\alpha)\otimes F(\beta)$ is equal to
\begin{multline*}
\dim\Hom_{U(n)}(F(\gamma),F(\alpha)\otimes F(\beta))=\dim\Hom_{U(n)}((F(\alpha)\otimes F(\beta))^*,F(\gamma)^*)=\\
\dim\Hom_{U(n)}(F(\gamma)^*,(F(\alpha)\otimes F(\beta))^*)=\dim\Hom_{U(n)}(F(\gamma)^*,F(\alpha)^*\otimes F(\beta)^*).
\end{multline*}
The last expression is equal to the multiplicity of $F(\gamma)^*$ in $F(\alpha)^*\otimes F(\beta)^*$.
\end{proof}

\section{The cases $(G,K)=(U(r, r+b), U(r)\times U(r+b)),
(Sp(r, \mathbb R), U(r))$}
\label{sect-su}
Theorem
\ref{thm1.1} is somewhat easier to prove
for these two 
cases.
\smallskip

{\it Proof of Theorem
\ref{thm1.1} for $(G,K)=(U(r, r+b), U(r)\times U(r+b))$.}
Denote 
$ \fg=\fgl(n, C)$, $n=2r+b$, which has one dimensional center.
 We use the abbreviations $\fg=\fgl(n, \mathbb C)
=
\fgl(n)
$,
$\fk=
\fgl(r)+ 
\fgl(r+b) 
$ with Cartan
subalgebra $\mathbb C^n=
\mathbb C^r +\mathbb C^{r+b}$
realized as diagonal matrices.
The relevant roots and linear
forms are
 $$\Delta^+=\{\e_i-\e_{j}; 1\le i<j\le  n\},\quad  
\Delta^+_n=\{\e_i-\e_{j}; 1\le i\le r, r+1\le j\le n\},$$
$$
\beta=\e_1-\e_n,\quad 
\gamma_i=\e_{r+b-i+1}
-\e_{r+b+i}, \quad i=1, 
\cdots, r,$$
$$
\rho=\frac 12
(n-1, \cdots, -(n-1)),\quad  
\zeta=\frac{1}{n} (
(r+b) 1^{r}  - r (1^{n}-1^{r})).
$$
Here we  write a linear form $\sum_{j=1}^n c_j\e_j$ 
on $\mathbb C^n$
as
$(c_1, \cdots,c_r|c_{r+1},  \cdots, c_n)$. 
The representations
of $\fk
$ are of the
form $F(\mu)\boxtimes F(\nu)$
with highest weight $(\mu_1,
\cdots, \mu_r|\nu_1, \cdots, \nu_{r+b})
$.
We use also the short-hand notation $1^j=(1, \cdots, 1, 0,\cdots, 0)
= \sum_{i=1}^j\e_i$ and 
$\mu^{-\ast}
=(-\mu_r, 
\cdots, -\mu_1)$
for any 
$\mu
=(\mu_1, 
\cdots, \mu_r)$.

The module
$L(\lambda)$
in question is 
$$
L(\lambda)=
L(-k\gamma_{1} 
 -2\zeta).
$$
This representation
is described
in \cite[Theorem 6.3]{KV}
(with their  $i=1$
and $k=2$) by the 
dual pair correspondence.
This is the last
reduction 
point in the family
of  $\lambda =-k\gamma_1+x\zeta$
by \cite{EHW}.
We now find the $K$-types.
The $K$-types
and their multiplicities 
can be obtained
by considering another
dual pair
$(
U(r)\times
U(r+b)
)\times 
(U(2)\times U(2))$
as one see-saw pair together
with the pair
 $U(r, r+b)\times U(2)$-correspondence.
 The central element
$cI_n =\diag
(cI_r, cI_{r+b})
\in U(r, r+b)$
acts on
$L(\lambda)$ as scalar 
$\det (cI_r)^{-2}
=c^{-2r}$; see \cite[Proposition 4.1]{KV}.
The two see-saw dual pairs
are
$$
\begin{matrix}
U(r, r+b) &\longleftrightarrow &U(2)\\
\bigcup & {} &\bigcap \\
U(r)
\times 
U( r+b) 
& \longleftrightarrow &U(2)\times U(2),
\end{matrix}.
$$
The dual pair correspondence
for  top pair is 
given in \cite{KV}
as
$$
L(\lambda )
\leftrightarrow (k, -k),
$$
and for the bottom pair
$(U(r)
\times
U(r+b)
)
\leftrightarrow
(U(2)\times U(2))$
is
a classical
result
\cite{NOT},
$$
F(0, \cdots, 0, -\mu_2, -\mu_1)\boxtimes F(
\nu_1, \nu_2, 0,\cdots, 0)
\otimes F(-2\zeta)_{\fk}
\leftrightarrow
F(\mu_1, \mu_2
)\boxtimes
F(-\nu_2,- \nu_1).
$$
From this we find
that any $K$-type
appearing in $L(\lambda)$
will be of the form
\begin{equation}
    \label{mu-nu-l-r}
(F(0, \cdots, 0, -\mu_2, -\mu_1
)_{\fgl(r)
}
\boxtimes 
F(\nu_1, \nu_2, 0\cdots, 0)
_{\fgl(r+b)})
\otimes F(-2\zeta)_{\fk},
\end{equation}
and it appears with equal multiplicity as
$F(k, -k)_{\fgl(2)}
$ in 
\begin{equation}
\label{gl2-tensor}
F(\mu_1, \mu_2)_{\fgl(2)}
\otimes
F(-\nu_2, -\nu_1)_{\fgl(2)},
\end{equation}
which is always one or zero.
It is one  exactly when
$$
(\mu_1-\mu_2) 
+(\nu_1-\nu_2) 
\ge 2k \ge
|(\mu_1-\mu_2)-(\nu_1-\nu_2)|,
$$
and
$$
-(\mu_1 +\mu_2)+
(\nu_1 +\nu_2)=0;
$$
there is an extra
general condition 
$(\mu_1-\mu_2) 
+(\nu_1-\nu_2)-2k =0\pmod 2$
which is automatic.
In other words,
the $K$-types in $L(\lambda)$
are precisely
those given in  (\ref{mu-nu-l-r}) where $(\mu_1, \mu_2)$
and 
$(\nu_1, \nu_2)$
satisfy the above conditions, and each
appears with multiplicity one; in particular
$L(\lambda)$ is of $K$-multiplicity free. 
\qed
\smallskip

Now consider
the  case $(Sp(r, \mathbb R), U(r))$.
We fix the relevant roots and the central character as follows:
 $$
 \Delta^+=\{\e_i\pm \e_{j}, 2\e_i; 1\le i<j\le  r\}, \quad 
 \Delta^+_n=\{\e_i+\e_{j}; 1\le i\le j\le  r\},  
 $$
$$
\gamma_j=2\e_{r-j+1},\, 1\le j\le r;\quad  \beta=\gamma_r=2\e_1,
$$
$$
\rho =(r, r-1, \cdots, 1), \quad
\zeta=1^r=(1, 1, \dots, 1).
$$
\smallskip

{\it Proof of Theorem
\ref{thm1.1} for $(G,K)=(Sp(r, \mathbb R), U(r))$.} 
We use
the dual pair
correspondence. It requires
somewhat more 
advanced results.
Let 
\[
\lambda=-k\gamma_1+x \zeta 
=-2k\e_r +x\zeta = (x, ..., x, x-2k)
\]
where $k> 0$
is an integer or a half integer, i.e.
$2k$ is an integer.
To find the K-types 
and their multiplicities
in $L(\lambda)$
we use
the see-saw dual pairs
(see e.g. \cite{NOT})
$$
\begin{matrix}
Sp(r, \mathbb R) &\longleftrightarrow &O(2)\\
\bigcup & {} &\bigcap \\
U(r) 
& \longleftrightarrow &U(2)
\end{matrix}
$$
the two
horizontal
pairs of groups
$(G, G')$
being
dual pairs in $Sp(2r, \mathbb R)$.
The restrictions
of the metaplectic
representation 
of $Sp(2r, \mathbb R)$
to the two 
pairs $(G, G')$
are multiplicity
free. Each of these restrictions 
is a direct sum 
of representations
of the form
$(\pi, \pi')$,
and there is 
a one-to-one correspondence $\pi\leftrightarrow
\pi'$.
The correspondence for the
upper row is $$
L(-\zeta -2k\epsilon_r)\leftrightarrow  2k,$$ 
where $2k$ denotes
the direct sum of
the two one-dimensional
representations
of $SO(2)$
with weights $\pm 2k$,
and
the correspondence for the lower row
is 
$$
F(-\zeta)
\otimes F(
-m_1\epsilon_r -m_2\epsilon_{r-1})
\leftrightarrow (m_1, m_2);$$ see \cite{KV, NOT}. 
From this we see that 
$F(-\zeta)
\otimes
F(-m_1\epsilon_r -m_2\epsilon_{r-1})$
appears in $L(-\zeta-2k\epsilon_r)$
if and only
if
$$
2k\le m_1-m_2,\quad 
2k=m_1-m_2 \pmod 2,
$$
in which case it has multiplicity one. We
write 
the K-types
using the Harish-Chandra root $\gamma_1=2\epsilon_r$ and 
the non-compact
root
$\delta=\epsilon_{r-1}+\epsilon_r
=\frac12(
\gamma_1 +\gamma_2)
$ 
as,
writing 
$m_1=2k+m_2+2l$, $m_2\ge 0 $, $l\ge 0$ integer
\begin{equation}
    \begin{split}
-\zeta-(2k+m_2+2l)\epsilon_r
-m_2\epsilon_{r-1}
&=
-\zeta -2k\epsilon_r
-2l\epsilon_{r}
-m_2(\epsilon_r+\epsilon_{r-1})
\\
&=
\lambda
-l\gamma_{1}
-m_2\delta,
    \end{split}
\end{equation}
which can be  rewritten
as 
$\lambda
-m_1\gamma_{1}
-
\frac {m_2}2(\gamma_1+\gamma_2).$
This proves our claim.
\qed
\smallskip

We give now another elementary
proof of the above statement
using the tensor product
decompositions
(\ref{ten-prod})
and (\ref{tensor-product-var}).
\smallskip

{\it Second proof of Theorem
\ref{thm1.1} for $(G,K)=(Sp(r, \mathbb R), U(r))$.}  
Let $k>0$ be an integer. 
We shall find the K-types in
$L(\lambda)$, 
$\lambda=-\zeta-2k\epsilon_r
$,  using 
Theorem
\ref{tensor-min} (1) which is
somewhat more elementary
than the theory of dual pairs. 
Consider first the case $k=1, \lambda=-\zeta -2\epsilon_r$.

Recall that
$L(-\frac 12\zeta)
=\sum_{m=0}^\infty
F(-\frac 12\zeta
-m\gamma_1)$. Note also $\gamma_1=2\epsilon_r,\ \gamma_2=2\epsilon_{r-1}$, and
\[
\lambda -
m_1\gamma_1 
-\frac{m_2}2(\gamma_1+\gamma_2)
= -\zeta -
(2+2m_1+m_{2})\epsilon_r
-m_{2}\epsilon_{r-1}.
\]
Consider the tensor product
\begin{equation}
    \label{F-tensor-F}
F(-\frac 12 \zeta
-2(m_1+l +1)\epsilon_r)
\otimes
F(-\frac 12 \zeta
-2l\epsilon_r)
\subset 
L(-\frac 12\zeta
)
\otimes
L(-\frac 12\zeta
)
\end{equation} 
We claim that the LHS
contains $$
F(-\zeta
-2(m_1+l +1)\epsilon_r
-2l 
\epsilon_{r-1}
).$$ Canceling the central characters, this is equivalent to 
\[
F(-2(1+m_1+l)\epsilon_r)
\otimes
F(-2l\epsilon_r)\supset F(
-2(m_1+l +1)\epsilon_r
-2l \epsilon_{r-1}).
\]
Passing to dual $\mathfrak u(r)$-modules and using Lemma \ref{dual tensor}, we see that this is further equivalent to
\[
F(2(1+m_1+l)\epsilon_1)
\otimes
F(2l\epsilon_1)\supset F(
2(m_1+l +1)\epsilon_1
+2l \epsilon_2).
\]
Now Littlewood-Richardson rule implies that this containment is indeed true; we simply put all the boxes with label 1, from the one-row diagram corresponding to $F(2l\epsilon_1)$, below the first row of the one-row diagram corresponding to $F(2(1+m_1+l)\epsilon_1)$; this is possible since $2l\leq 2(1+m_1+l)$.

We now prove that the projection 
of the $K$-type 
$F(-\zeta-2(m_1+l +1)\epsilon_r-2l 
\epsilon_{r-1}
)
\subset
L(-\frac 12\zeta
)
\otimes
L(-\frac 12\zeta)$ 
under 
\[
J_k=J_1:
L(-\frac 12\zeta
)
\otimes
L(-\frac 12\zeta)
\to L(\lambda)
=L(-\zeta -2\epsilon_r)
\]
is non-zero.
Up to a non-zero
constant $J_1 f(z) =(\partial_z -\partial_w)f(z, w)|_{z=w}$ where $\partial_z f=\sum_{ij}
(\partial_{ij}f) dz_{ij}$
is the differentiation
in $z$ taking value
in $(\fp^+)'$
as differential forms.
Disregarding
the central character,
the highest weight
$
-2(m_1+l +1)\epsilon_r
-2l 
\epsilon_{r-1}
$ of the above K-type is a sum
of the highest
weight $-2(m_1+l +1)\epsilon_r$ 
in the first tensor 
factor and the weight 
$-2l 
\epsilon_{r-1}$
in the second tensor factor in
(\ref{F-tensor-F}).
Now the highest weight
vector
in the first tensor factor
$F(-\frac 12\zeta -2
(m_1+l +1)\epsilon_r)
)
$
is
$z_{rr}^{m_1+l+1}$,
whereas a weight
vector of weight
$-2l 
\epsilon_{r-1}$
in 
the second tensor factor $F(-\frac 12\zeta-2l\epsilon_r)$
is of the form
$w_{r-1, r-1}^{l}$. Here we
use the realization
of $Sp(n, \mathbb R)/U(n)$ as the
bounded symmetric domain of symmetric matrices.
Thus the highest
weight vector
in
$
F(-\zeta-
2(m_1+l+1)\epsilon_r
-2l \epsilon_{r-1}
)
$
is
of the form
$$
f(z, w) 
=z_{rr}^{m_1+l+1}
w_{r-1, r-1}^{l}
+ \text{Rest} 
$$
where the term
Rest 
is a sum of monomials
which have lower degree in $z_{rr}$ than 
$z_{rr}^{m_1+l+1}$.
We perform
now the differentiation
$J_1f$. It has
a form
\[
J_1f(z)
= 
(m_1+l+1)z_{rr}^{m_1+l}
z_{r-1, r-1}^{l} dz_{rr}
+ \text{Rest}.
\]
It is clear
that the terms in Rest
all have lower degree in
$z_{rr}$. Thus $J_1f\ne 0$, that
is, the $K$-type 
$F(-\zeta
-2(m_1+l+1)\epsilon_r
-2l \epsilon_{r-1}
)$
will indeed appear
in $L(\lambda)$.

The result for $k=1$
can be summarized
as that there 
is a highest
weight vector $f_1(z, w)$
of weight $-(2+2m_1+m_2)
\epsilon_r-m_2\epsilon_{r-1}$
 with degree of vanishing along the diagonal $z=w$ being exactly $k=1$.
For any $k>1$
the element $
(z_{rr}-w_{rr})^{k-1}
f_1(z, w)$
has then exact
vanishing degree $k$
and is a highest weight 
vector since 
$(z_{rr}-w_{rr})^{k-1}$
is a highest weight 
vector, and
the pointwise product of two highest weight vectors is again a highest weight vector. Moreover, 
it is clear that
this element is
also in $
L(-\frac 12\zeta)
\otimes
L(-\frac 12\zeta).
$

Finally all $K$-types in 
$L(-\frac 12\zeta)
\otimes
L(-\frac 12\zeta)$
have appeared in
$L(-\zeta-k\gamma_1)$
by the above procedure using 
Taylor expansion
for $k=0, 1, 2, \cdots$. Thus
$L(-\zeta-k\gamma_1)$
contains precisely 
those $K$-types that 
appeared in this way, 
by the decomposition  (\ref{ten-prod}).

Now let $k=k'+\frac 12$ be a half integer, $k'\ge 0$.  We realize $
L(-\zeta -k\gamma_1)
=L(-\zeta -(2k'+1)\epsilon_r)
$
in the space
$L(-\frac 12\zeta-\epsilon_r)
  \otimes 
L(-  
\frac12\zeta)  
$ given in  (\ref{tensor-product-var}), 
as the space
of sums of functions
of the form
$A(z)B(w)$
where $A(z)$ are
 $\mathbb C^r$-valued
 polynomials and $B(w)$
 are scalar-valued
 polynomials.
The same method
above can be carried
out in this case:
The function 
$(z_{rr}-w_{rr})^{k-1}f_z(z, w)$
is replaced by
$$
\tilde f(z, w)=
(z_{rr}-w_{rr})^{k'}
(z_{rr}^{m_1+l}
w_{r-1, r-1}^l e_r + Rest).
$$
The relevant functions
in $z$ are in the space
$L(-\frac 12 \zeta -\epsilon_r)$
since 
$L(-\frac 12 \zeta -\epsilon_r)$
is multiplicity free,
and the highest
weight vectors in 
the space are of the form
$z_{rr}^n e_{rr}$.
The image of 
$\tilde f(z, w)$
under the intertwining
map (\ref{bi-lin-int-v})
is non-zero
by the vanishing
property along
the diagonal $z=w$.
The rest of the proof
is exactly the same.
\qed
\smallskip

\begin{remark}
Note that when $k=0$
the $K$-types
in $L(-\zeta)$
are of the form
$-\zeta-m_1\gamma_1
-m_2\gamma_2
=-\zeta-2m_1\e_1
-2m_2\e_2
$
and can not be obtained
from the above
result by formally
putting $k=0$. This
is due to the different
behavior of the dual pair
correspondence.
\end{remark}

\section{The case $(G,K)=(SO_0(2, m), SO(2)\times SO(m)), m\ge 4$}

We notice that 
$SO_0(2, m)$ 
can only appear  in the dual
pairs $SO(2, m)\times Sp(n,  \mathbb R)$
where both factors are non-compact. Thus
it is generally rather difficult
to use see-saw dual pairs to
find $K$-types as
the decomposition of
the metaplectic representation under 
$SO(2, m)\times Sp(n,  \mathbb R)$ is already
quite demanding.

The root system is $\Delta=\Delta_c\cup \Delta_n$
and the relevant
roots are as follows: 

\noindent If
 $m$ is even, 
$$
\Delta_c^+ =\{\e_i\pm \e_j, 1\le i < j\le 
\dfrac m2
\},\qquad 
\Delta_n^+ =\{\e_0 \pm \e_j, 1\le  j\le 
\dfrac m2\},$$
with
$$\rho = (
\dfrac m2, \dfrac m2 -1, \cdots, 0).
$$
If $m$ is odd,
\begin{eqnarray*}
\Delta_c^+ &=&
\{\e_i\pm \e_j, 1\le i < j\le \frac{m-1}2\}
\cup
\{\e_i, 1\le i \le \frac{m-1}2
\}
,\\
\Delta_n^+ &=&\{\e_0 \pm \e_j, 1\le  j\le 
\frac{m-1}2
\}\cup\{\e_0\},
\end{eqnarray*}
with
$$
\rho = (
 \frac{m-1}2, 
\frac{m-1}2-1, \cdots, 0)+\frac 12 (1, 1, \cdots, 1).
$$
In all cases,
$$
\gamma_1=\e_0-\e_1<
\gamma_2=\beta=\e_0+\e_1,
\quad \zeta=\dfrac 12(\gamma_1+\gamma_2) =\e_0,$$
the root multiplicity
$a$  in (\ref{ab-def})
and genus (\ref{genus}) $p$
are  $a=m-2$,  $p=m$.

We write a $\fk$-weight $\lambda$ as
\[
\lambda =x\e_0  + x_1\e_1 +\dots +
x_{[\frac{m}2]}\e_{[\frac{m}2]
}
=(x|  x_1, \dots,  x_{[\frac{m}2]}),
\]
and the corresponding
representation of $\fk=
\fso(2, {\mathbb C}) + \fso(m, {\mathbb C})$ as 
\[
F(\lambda) = \mathbb C^{x} \otimes 
E^{
( x_1, \dots,  x_{[\frac{m}2]}
)
}.
\]
The module $L(\lambda)$
we consider has $\fg$-highest
weight $\lambda = -k\gamma_1 +x\zeta$,
our $x$ corresponds to $z$ in \cite{EHW}
by  $z=x + m-1$, and there is only 
one reduction point at $z=1$, i.e., at $x= -(m-2)$ with $\lambda = -k\gamma_1 -(m-2)\zeta$.

%\begin{theorem}\label{so-m-2-case}
%Let $m\ge 4$. The $\fk$-types 
%  appearing 
%  $L(\lambda)$ at the last reduction point 
%    $\lambda=-k\gamma_1 -(m-2)\zeta$
%    are of the form   $$
%    \mu=\lambda -m_1\gamma_1 -m_2\gamma_2 -j (\e_1-\e_2),  \quad\ m_1\ge m_2\ge 0, \ \ 0\le j\le\min\{k,  \frac{m_1-m_2}2\}, $$
%with multiplicity one.
%\end{theorem} 

\smallskip

 {\it Proof of Theorem
\ref{thm1.1} for $(G,K)=(SO_0(2, m),
SO(2)\times SO(m))$.}
We use \cite[Theorem 5.2]{EJ}
  and 
some known   results 
  on tensor products of $SO(n)$-representations \cite{Okada-98, St}.

  For notational convenience we write 
$$-m_1\gamma_1 -m_2 \gamma_2
=-m_1(\e_0-\e_1) - m_2 (\e_0+\e_1)
=-(d+2c) \e_0+d\e_1,
$$
with 
$c=m_2, d=m_1-m_2 $,
and  as $\fso(2) + \fso(m)$-representations, 
$$
F(-m_1\gamma_1- m_2\gamma_2) =
\mathbb C^{-(d+2c)} \otimes  E^{(d)}, (d)=(d, 0, \dots, 0).
$$

  The point $\lambda=-k\gamma_1+x\zeta$, $x=-(m-2)$, is
  the only reduction point, and thus $L=L(\lambda)=N/M$, where
  $M$
  is the maximal submodule of $N$.
Now a resolution of $L$ is given
in  \cite[Theorem 2]{EW} 
where each module is generated
by a fixed $K$-type over $S(\mathfrak p^-)$
and the length 
of the resolution
(excluding the first and
the last $0$ terms) 
is  $|\Delta_\lambda^+
\cap \Delta_n^+| +1$, where 
$\Delta_\lambda^+$ is a certain
subroot system of $\Delta$.
In the present case it is
easily checked that $\Delta_{\lambda}^+$ 
is  
the singleton
$\{\gamma_2\}$,
and 
the corresponding
 subset $\mathcal W_\lambda^{\mathfrak k, i}$
is a single element $\{s_{\gamma_2}\}$
for $i=1$
and otherwise empty.
The generating $K$-type
for    $M$ is 
given \cite[Theorem 2]{EW} by 
$\tilde\lambda
=(s_{\gamma_2}(\lambda +\rho))^+ -\rho
$
where $\mu^+$ is
the $\mathfrak k$-dominant weight
and is in the orbit of $\mu$ under the Weyl group of $\Delta_c$,
and more explicitly,
$$
\tilde\lambda=\lambda-\gamma_2=
 -k\gamma_1 - (m-2) \zeta -\gamma_2
  =(-m+1-k; k-1, 0, \dots, 0).
  $$ 
  Alternatively,
this can be obtained by \cite[Theorem 5.2]{EJ},
which states
 that the maximal submodule $M$ is  generated by the PRV component of
 $F(\lambda) \otimes \fp^-$ which in our case is indeed $F(\tilde{\lambda})$.

We write
  $$
  \tilde \lambda
  =   \lambda-\gamma_2
  = ( \lambda  +\gamma_1) 
-( \gamma_1 +\gamma_2)
    =-(k-1)\gamma_1 -m \zeta
    $$
 which is on the shifted line 
of our $\lambda=-k\gamma_1 -(m-2)\zeta$ 
    with
    the  parameters $(k, m-2)\mapsto (k-1, m)$, namely $k$ replaced by $k-1$,
    and $m-2$ by $m$, and
    is in the continuous part
    of the analytic continuation
    of the line $-(k-1)\gamma_1+x\zeta$. Thus
    $   M=S(\fp^-)\otimes F(\mu)$ is the whole Verma module.
    We compute the corresponding $\fk$-types
    in $N$ and in $M$.
We have    \begin{equation*}
      \begin{split}
    N&=S(\fp^-)\otimes F(\lambda)=
    \oplus_{c, d \ge 0} F(c+d, d) \otimes  F(\lambda)\\
    &
    =     \oplus_{c, d \ge 0}
    (\mathbb C^{(2c+d)} \otimes E^{(d)})\otimes
    (\mathbb C^{-k-p+2} \otimes E^{(k)})
    \\
    &=
    \oplus_{c, d\ge 0 } \mathbb C^{-p+2 -k -2c-d}
\otimes(    \oplus_{y: (d, k) } E^{y})
      \end{split}
    \end{equation*}
    where $\oplus_{y: (d, k) }$
    stands for the sum
    $$
    E^{(d)} \otimes  E^{(k)} =
    \oplus_{y: (d, k) }
        E^{y}
    $$
    multiplicities counted.
Likewise,
    $$
    M=S(\fp^-)\otimes F(\mu)=
 \oplus_{c, d\ge 0} \mathbb C^{-m+1 -k -2c-d}
\otimes(  \oplus_{y: (d, k-1) }      E^{y}).
    $$
    To find the highest weights in $L=N/M$ we match the central characters in the sum of
    $N$ and $M$,    rewrite $d\ge 0$ as $d-1$, $d\ge 1$, to get
    $$
    M=\oplus_{c\ge 0, d\ge 1} \mathbb C^{-m+2 -k -2c-d}
\otimes(  \oplus_{y: (d-1, k-1) }      E^{y}).
    $$

    The space $E^{(d)}$
    can be realized as the space
    of spherical harmonics of degree $d$ on $\mathbb R^{m}$ and
    the tensor product decomposition of
    $E^{(d)}\otimes
    E^{(k)}   $
    of such representations
    is known; indeed
 the tensor product decomposition
    $E^{(d)}\otimes
    E^{(k)}   $
    is multiplicity free and 
    it can  obtained
    using the general result  \cite[Corollary 2.16]{St}.    
        We have
    \begin{equation}
      \label{Stembrige}
          \begin{split}
E^{(d)}\otimes 
E^{(k)}
&= \oplus_{i, j\ge 0, i+j\le \min\{k, [\frac{d+k}2]
\}} E^{(d+k-2i-j, j, 0, \dots, 0)}
\\
&
=
\oplus_{j=0}^{\min\{k, 
[\frac{d+k}2]\}
} \oplus_{i=0}^{\min\{k, 
[\frac{d+k}2]
\} -j} 
 E^{(d+k-2i-j, j, 0, \dots, 0)}.
          \end{split}
         \end{equation}
  and correspondingly 
  \begin{equation}
    \begin{split}
E^{(d-1)}\otimes 
E^{(k-1)}
&= \oplus_{i, j\ge 0, i+j\le \min\{k-1, 
[\frac{d+k}2]-1
\}} E^{(d+k-2i-2-j, j, 0, \dots, 0)}
\\
&=
\oplus_{j=0}^{\min\{k-1,
[\frac{d+k}2]-1\}}
\oplus_{i=0}^{\min\{k-1, 
[\frac{d+k}2]-1\} -j} 
 E^{(d+k-2i-2-j, j, 0, \dots, 0)}    
    \end{split}. 
  \end{equation}

Thus $L=N/M$ contains precisely
    \begin{equation*}
      \begin{split}
 L&=N/M=\oplus_{c, d\ge 0 } \mathbb C^{-m+2 -k -2c-d}
\oplus_{j=0}^{\min\{k, 
[\frac{d+k}2]\}
}  
 E^{(d+k-j, j, 0, \dots, 0)}\\
&
=
\oplus_{c, d\ge 0, j\le k, 0\le j\le \frac{d}2\}} F(-a\zeta  -k\gamma_1
 -(c+d)\gamma_1 -c\gamma_2 -j(\e_1-\e_2) ) \\
& =
\oplus_{m_1\ge m_2\ge 0, 0\le j\le k, j\le \frac{m_1-m_2}2\}} F(
\lambda  -m_1\gamma_1 -m_2\gamma_2 -j(\e_1-\e_2) )
          \end{split}
\end{equation*}

This proves our claim.
  \qed

  \begin{remark}
  When $m=4$
  our result
  can also be obtained by using the fact
  that $\mathfrak{so}(2, 4)
  =\mathfrak{su}(2, 2)
  $ and
the  result in 
Section  \ref{sect-su}
  for   $\mathfrak{su}(r+b, r)$.
  Indeed, the highest
  weight $\lambda-m_1\gamma_1 -m_2\gamma_2-j(\epsilon_1-\epsilon_2)$
  of $\mathfrak{so}(2,
  \mathbb C)
  + \mathfrak{so}(4, 
  \mathbb C) $
considered here corresponds 
  to the highest weight
  $(-m_2, -(m_1+k); 
  m_1+k-j, m_2+j$
  of $\mathfrak{gl}(2,\mathbb C)
  +\mathfrak{gl}(2, \mathbb C)$  considered  there
  and it is easy to verify that 
  the conditions 
coincide.
   \end{remark}

\medskip

%%
%%The identity
%    (\ref{Okada-98})
%can be reformulated 
%  as 
%%  \begin{equation}
%% \label{tensor-key}
%%    F(-l\gamma_1)\otimes 
%%    F(-m\gamma_1) 
%%    =\oplus_{i+j\le l} F(-m\gamma_1 -(l-i)\gamma_1 + i\gamma_2 - j\alpha_0 )
%%  \end{equation}
%%  for $l\le m$, where %%$\alpha_0=\e_1-\e_2$
%%  is the simple compact %%root connected to %%$\gamma_1=\e_0-\e_1$.
%%  This result will be %%needed in the coming %%subsections.
%% There is a difference %%between tube and non-tube %%domains; see Lemma
%% 12.8 for e6-case.

\section{The case
$(G, K)=(SO^*(2n), U(n))$, $n\ge 4$}
\label{so*}

Let $r=[\dfrac n2]$ be the real rank of $G$. The complexified Lie algebras are 
$\fg=\fso(2n, \mathbb C)$ and $\fk=\fg\fl(n,\mathbb{C})$, 
with the  following choices of the relevant roots and central
character $\zeta$:
 $$\Delta^+=\{\e_i\pm \e_{j}; 1\le i<j\le  n\},\quad  
 \Delta^+_n=\{\e_i+\e_{j}; i=1, \cdots, n\},
 $$
$$
\beta=\e_{1}+\e_2,\ 
\gamma_1=\e_{n-1}+\e_{n},\ 
\gamma_2=\e_{n-3}+\e_{n-2},\  
\cdots, 
\gamma_r=\e_{n-2r+1}+\e_{n-2r+2},
$$
$$
\zeta=\frac{1}{2}1^n=\frac{1}{2}\sum_{j=1}^n\epsilon_j, \quad  \rho=\sum_{j=1}^n (n-j)\epsilon_j.
$$
The relevant constants
in (\ref{ab-def})-
(\ref{genus}) in this case are $p=2n-2$, and $(a, b)=(4, 0)$ if $n$ is even and
$(a, b)=(4, 2)$ if $n$ is odd.

%%
%%\begin{theorem}
%%our \label{thm-skew-ktype}For any positive integer $k$, the $K$-types of the unitary highest weight module $L(\lambda)$, $\lambda=-k\gamma_1-4\zeta$, are all of the form $F(\lambda)\otimes F(0,\dots,0,-a_4,-a_3,-a_2, -a_1)$
%%with\begin{equation}\label{cond a}\begin{split}(a_1, a_2, a_3, a_4)=(m_1, m_1-j, m_2+j, m_2),\quad 2j\le m_1-m_2\end{split}\end{equation}with multiplicity 1.\end{theorem}

\smallskip

{\it Proof of Theorem
\ref{thm1.1} 
for $(G,K)=(SO^\ast(2n),
U(n))$.}
 We use the realization
 (\ref{tensor-min})
  of $L(-4\zeta-k\gamma_1)$
as a subspace of
$L(-2\zeta)\otimes 
L(-2\zeta)
$, with
\[
L(-2\zeta)=
\sum_{\mu \ge
0}
F(-2\zeta)\otimes 
F(0, \cdots 0, -\mu, -\mu)
\]
as  in
(\ref{last-W-pt}).
Using Lemma
\ref{lem LR2} and Lemma \ref{dual tensor}, we see that any 
K-type in 
$L(-2\zeta)\otimes 
L(-2\zeta)$ is of the form
$$F(-4\zeta)\otimes
F(0, \cdots, 0, -a_4, -a_3, -a_2, -a_1),$$
with $
F(a_1, a_2, a_3, a_4, 0, \cdots, 0)$
being  a component in 
\begin{equation}
\label{mu-mu'} 
(\oplus_\mu F(\mu, \mu, 0, \cdots, 0))
\otimes 
(\oplus_{\mu'}F(\mu', \mu', 0, \cdots, 0)),
\end{equation}  namely it has
maximal length $4$
with the condition $a_1-a_2=a_3-a_4$.

Recall that the domain $D$ is realized
as  {a} matrix ball
in the Jordan triple
$V=\{z\in M_{n, n};
z=-z^t\}$
of skew-symmetric complex $n\times n$-matrices.
Now let $V_2\subset V$ be the
Jordan triple of $4\times 4$-skew-symmetric matrices realized as the Jordan subtriple of V consisting of skew-symmetric matrices of block diagonal form $\text{diag}(0, z)$, $z\in V_2$.
The symmetric pair corresponding to $V_2$ is
$(\fg'_0, \fk'_0)=(\mathfrak {so}^\ast(8),
\mathfrak u(4))$
which is also $(\fg'_0, \fk'_0)
=(\mathfrak{so}(6, 2), \fso(6)+\fso(2))$,
and the corresponding symmetric space
$D':=D\cap V_2$
is the symmetric
space of $SO(6, 6)$ considered
in the previous section.
The correspondence between 
$\mathfrak{gl}(4, \mathbb C)
=\mathfrak{u}(4)^{\mathbb C}
$-roots and 
 $\mathfrak{so}(2,
 \mathbb C)+
\mathfrak{so}(6, \mathbb C)$-roots is
$$
\epsilon_1^{\mathfrak{u}(4)}=
\frac 12 (\epsilon_0 -\epsilon_1-\epsilon_2-\epsilon_3),
\epsilon_2^{\mathfrak{u}(4)}
=\frac 12 (\epsilon_0 -\epsilon_1 +\epsilon_2 +\epsilon_3),
$$
$$
\epsilon_3^{\mathfrak{u}(4)}
=
\frac 12 (\epsilon_0 +\epsilon_1 -\epsilon_2 +\epsilon_3),
\epsilon_4^{\mathfrak{u}(4)}
=\frac 12 (\epsilon_0 +\epsilon_1 +\epsilon_2 -\epsilon_3).
$$

We use the notation as in  Subsection
  \ref{Sub-sym-pair}
and 
compare with the corresponding
decomposition 
$L(-2\zeta)
\otimes 
L(-2\zeta)
=\sum_{k=0}^\infty
L(-4\zeta -k\gamma_1)
$
with 
$L'(-2\zeta)
\otimes 
L'(-2\zeta)
=\sum_{k=0}^\infty
L'(-4\zeta -k\gamma_1)
$
for ${\mathfrak g'}$.
Firstly,
the $K'$-types
 for the $(\fg', K')$-module
$L(-4\zeta-k\gamma_1)'$
are $F(-4\zeta)'
\otimes
F(\mu)'
$,
$$
\mu =(-m_2, -(m_2+j), -(m_2-j), -m_1),
$$
by our Theorem
\ref{thm1.1} 
in the case of $SO_0(2, 6).$
Secondly,
all the
$
K'$-
highest weights
in 
$L(-4\zeta-k\gamma_1)'$
are also
$
K$-highest weights by  Lemma \ref{reduct-subdom}. Thirdly, all
 $K$-types 
in $L(-2\zeta)
\otimes
L(-2\zeta)$
are covered
by the above counting, for all $k\ge 0$, since 
the decomposition
of (\ref{mu-mu'}) is
independent of
$n\ge 4$. 
This completes the proof.
\qed
\begin{remark} We remark
that we actually
proved that any
$K$-type appearing in $L(-4\xi-k\gamma_1)$ is obtained
by the projection of
a $K$-type in 
$F(-2\xi+\mu)
\otimes F(-2\xi -k\gamma_1) 
\subset L(-2\xi)
\otimes F(-2\xi -k\gamma_1) 
\subset  L(-2\xi)
\otimes L(-2\xi) 
$ through the bi-linear differential operator $J_k$. 
\end{remark}

\begin{corollary}
\label{cor fd}
The representation 
  $(k, k, 0, 0)$ of $Sp(4)$ 
  appears
  with multiplicity one in the representation 
 $\mu=(\mu_1, \mu_2, \mu_3, \mu_4, 0, \cdots, 0)$ of $U(8)$, where 
 $\mu$ corresponds
 to the dual of a
 K-type of $L(-k\gamma_1 -4\zeta)$
 under the $U(n)\times U(8)$
 duality.

\end{corollary}
\begin{proof}
We recall
 \cite{NOT}
the
dual pair correspondence
for the
metaplectic
representation
of $Sp(8n, \mathbb R)$
restricted 
to the 
see-saw pairs
$SO^\ast (2n)
\times Sp(4)$,
and $U(n)\times U(8)$. The representation
$L(\lambda)$, 
$\lambda =-k\gamma_1-4\zeta$ of $SO^\ast (2n)$ 
 corresponds 
to the representation 
$(k,k, 0, 0)$ of $Sp(4)$, and
the representation
$(\mu_1, ..., \mu_4,
0, \cdots, 0)$ of $U(n)$ 
corresponds to the representation 
$(\mu_1, ..., \mu_4, 0\cdots, 0)$ of  $U(8)$.
The result 
now follows
from Theorem \ref{thm1.1}  and general
principles
of see-saw pairs \cite{NOT}.
\end{proof}

\section{The case $(G,K)=(E_{6(-14)}, Spin(10)\times SO(2))$ }
\label{e-6-case}

The Dynkin diagram of $E_6$ is
\[
\begin{tikzcd}[arrows=-]
  E_6{:} &[-2em]
  \stackrel{1}{\circ} \arrow[r] &
  \stackrel{3}{\bullet} \arrow[r] &
  \stackrel{4}{\bullet} \arrow[r]\arrow[d]
  &\stackrel{5}{\bullet} \arrow[r] &
  \stackrel{6}{\bullet} \\
 & && \stackrel{\bullet}{2}  &&
\end{tikzcd},
\]
the black roots  $\{\alpha_2, \cdots, \alpha_6\}$
being compact roots and $\alpha_1$ non-compact.
The compact and non-compact positive roots are 
$$\Delta_c^+=\{\e_i\pm \e_j; 5\ge i>j\ge 
1\},
$$
$$\Delta_n^+=\{\dfrac 12 (\sum_{i=1}^5 
(-1)^{\nu_i} 
\e_i -\e_6-\e_7 +\e_8); 
\text{$\sum_{i=1}^5 {\nu_i} $ is even}\},
$$
with
$$
\alpha_1=
\dfrac 12 (1, -1, 
\cdots,-1; -1  -1, 1), \, \alpha_2 =\e_1+\e_2,\, \alpha_j=\e_{j-1}-\e_{j-2},\,  3\le j\le 6, 
$$
$$\gamma_1=\alpha_1<\gamma_2 =
\dfrac 12 (-1, 1, 1, 1, -1; -1,  -1, 1),$$
$$
\beta=\beta_1
=\dfrac 12 (1, \cdots, 1; -1, -1, 1)>\beta_2
=\frac 12(-1, \cdots, -1, 1; -1, -1, 1)
$$
$$
\rho=(0, 1, 2, 3, 4; -4, -4, 4), \zeta=
\dfrac 23 (-\e_6-\e_7+\e_8);
$$
the simple root connected to $\gamma_1$ is $\alpha_3=\e_2-\e_1$
with the Killing form $(\beta, \rho)=11=p-1$.
The Jordan characteristic  $(r, a, b)=(2, 6, 4)$ and
the
genus $p=12$;
the  root 
multiplicity 
of  $\frac{\gamma_2-\gamma_1}2$ is $6$,
and 
the multiplicity
of the roots 
$-\frac{\gamma_1
}2$ and $-\frac{\gamma_2}2$
is equal to 4.

We shall consider the maximal
tube domain defined by 
the sub-pair
$$(\fg',  \fk')=
  (\fso(10, \mathbb C),
  \fso(2, \mathbb C)+\fso(8, \mathbb C))$$
  of $(\fg, \fk)$.
  Now 
  the spaces $\mathfrak k_{12}$ and 
    $\mathfrak k_{12}'$ in Lemma \ref{root-sps} 
    are of dimension
     $12$,
    and $\mathfrak k_{11}' +
    \mathfrak k_{22}'\subset \ft'\subset \mathfrak t$ is in the Cartan subalgebra 
     of $\fk'$, so 
    Lemma \ref{root-sps} is now verified by counting
    dimensions.

We consider  $\lambda=
-k\gamma_1 + x\zeta$, 
$x=1-p +z=-11 +z$. 
Note
here that $(\gamma_1, \beta)=0$.
The root systems
$Q(\lambda)
$
and $R(\lambda)$
are $Q(\lambda)=R(\lambda)
=
\fsu(1, 5)$, namely
it corresponds to the
Hermitian symmetric
space $SU(1, 5)/U(5)$ in $\mathbb C^5$
with $SU(1, 5)/U(5) \times
SU(1, 1)/U(1)$ being
a maximal Hermitian symmetric sub-domain in
the domain $E_{6(-14)}/Spin(10)\times SO(2)$.
The last 
reduction point is at $z=5$, $x=-6=-a$,
and $\lambda=-k\gamma_1 -6\zeta$. 

We find now the $K$-types
in $L(-a\zeta -k\gamma_1)$.

\smallskip
 {\it Proof of Theorem
\ref{thm1.1} for $(G,K)=(E_{6(-14)}, Spin(10)\times SO(2)
)$.}
The space $L(-6\zeta -k\gamma_1) $ will be realized
as one  irreducible component 
in the tensor product
decomposition (\ref{ten-prod}).
We consider 
the tensor products of the $K$-types, and we claim that
\begin{equation}
\label{F-l-times-m}
F(-l\gamma_1)\otimes F(-m\gamma_1)=
\sum_{i+j\le l}
F(-(l+m)\gamma_1 -i (\gamma_2-\gamma_1)-j\alpha_3).
\end{equation}
We use the short-hand notation
$$1^5_- := (-1, 1, 1, 1, 1),\quad 
(\dfrac 1 2)^5_- := \dfrac 12(-1, 1, 1, 1, 1),
$$
and then
$$
-\gamma_1 =(\dfrac 12)^5_- -\dfrac 34\zeta.
$$
A general weight of the form $-m_1\gamma_1-m_2\gamma_2-j\alpha_3$
can then be written as
$$
-m_1\gamma_1-m_2\gamma_2-j\alpha_3
=
c\e_5+  d (\dfrac 12)^5_- - j\alpha_3 +
( -\dfrac 34 d -\dfrac 32 c )\zeta
$$
and 
the corresponding representation is
$$
F(-m_1\gamma_1-m_2\gamma_2-j\alpha_3)
=E(c\e_5+  d (\dfrac 12)^5_- - j\alpha_3)
\boxtimes
F(
-\dfrac 34 (d+2c) )
$$
as representation of $\fso(10)+\fso(2)$, with $c=m_2,\ d=m_1-m_1$,
where $F(-\dfrac 34 (d+2c)
\zeta)
$ is  the one-dimensional
representation of $\fso(2)=\mathbb RZ$.

The tensor product
decomposition
of $F(-l\gamma_1)\otimes
F(-m\gamma_1)$, $\ l\le m$,  is given by
$$
F(-l\gamma_1)\otimes
F(-m\gamma_1)
=(
E( (\dfrac l2  )^5_-) \otimes 
E( (\dfrac m2)^5_-)
) 
\boxtimes 
F(-\dfrac 43 (l+m)\zeta),
$$
and 
\begin{equation*}
  \begin{split}
&\quad E( (\dfrac l2  )^5_-) \otimes 
E( (\dfrac m2 )^5_-)\\
&=\sum_{0\le p\le  q\le l} E(-p-\dfrac  12 (m-l),
p+\dfrac  12 (m-l),   q +\dfrac  12 (m-l),
q+\dfrac  12 (m-l),   \dfrac  12 (l+m))
\\
&=\sum_{0\le p\le  q\le l}
E(
(q +\dfrac 12(m-l))
1^5_- + (l-q)\e_5 -(q-p)\alpha_3
);    
  \end{split}
\end{equation*}
see \cite[Theorem 2.5(3)]{Okada-98}
and note that
here we are using a reversed  ordering.
This is precisely our claim
(\ref{F-l-times-m})
with $i=l-q,\ j=q-p$, $\ i+j=l-p\le l$,

Now we can 
consider the 
sub-symmetric pair
$(\fg'_0,  \fk'_0)=
  (\fso(2, 8),
  \fso(2)+\fso(8))$ in
  $(\fg_0, \fk_0)$
and use the
same steps
 as
in the proof of 
Theorem \ref{thm1.1}
for $(G, K)= (SO^*(2n), U(n))$.
\qed
\smallskip

\begin{rema+}
  It is possible to  prove the above results by first finding all
  $\fk$-types in 
  $
  L(-\dfrac a2\zeta)   \otimes 
  L(-\dfrac a2\zeta)=
  \sum_{k\ge 0}
  L(-a\zeta -k\gamma_1) $
  using the result of Okada \cite{Okada-98},
  and then performing induction to find  that  those $\fk$-types
  appear in   $L(-a\zeta -k\gamma_1) $.
  \qed
\end{rema+}

\section{The case $(G,K)=(E_{7(-25)}, E_6\times SO(2))$ }
\label{e-7-case}

%The $e_7$ case is somewhat
%easier than the $e_6$ case and we treat it first.
%We shall prove
%that some of the singular representation
%for $e_6$ can be obtained
%by co-normal differentiation
%of the minimal representations
%of $e_7$ via the embedding
%of the symmetric domain 
%of $E_{6(-14)}$ into that 
%of $E_{7(-25)}$ and we recover some multiplicity
%result in \cite{EHW} 
%for vector-valued holomorphic representations. 

The Dynkin diagram of $\fg=E_7$ consists of the following simple roots $\alpha_j$, $1\le j\le 7$,
   \[
\begin{tikzcd}[arrows=-]
  E_7{:} &[-2em] \stackrel{1}{\bullet}\arrow[r] &\stackrel{3}{\bullet}
  \arrow[r] &
  \stackrel{4}{\bullet}  \arrow[r] \arrow[d] &
  \stackrel{5}{\bullet} \arrow[r] &\stackrel{6}{\bullet}
  \arrow[r] &\stackrel{7}{\circ} \\
 & & &\stackrel{\bullet}{2} & && 
\end{tikzcd},
\] 
with $7$th root
$\alpha_7=\gamma_1$ being
the simple non-compact root, the six black roots 
forming the root system 
$\Delta_c^+=\Delta^+(e_6)$
of $e_6$ given in the previous section,
and with  non-compact roots
$$
\Delta_n^+=\{-\e_7+\e_8, \pm \e_i +\e_6; 1\le i\le 5 \}
\cup \{
\dfrac 12(
\sum_{i=1}^5 (-1)^{\nu(i)} \e_i
+\e_6-\e_7 +\e_8
); \text{$\sum_{i=1}^5 {\nu(i)}$ is odd}\},
$$
$$
\gamma_1
=\alpha_7=\e_6-\e_5,\ 
\gamma_2=\e_6+\e_5,\ 
\gamma_3=\beta
=\e_8-\e_7,
$$
$$
\rho
=(0, 1, 2, 3, 4, 5; -\frac{17}2,
 \frac{17}2), 
$$
$$\zeta=\dfrac 12(
\gamma_1 +
\gamma_2 +
\gamma_3) =\e_6 + \dfrac 12(-\e_7 + \e_8);
$$
here $(r, a, b)=(3,  8, 0)$
and the
genus $p=18$, $p-1=(\rho, \beta)$. See \cite{EHW}.
The root multiplicities of
$\frac{\gamma_{i}-\gamma_j}2$, $i>j$,
are $8$,
and sum space
$(\sum_{i=1}^3\mathfrak k_{jj})\cap \mathfrak n_{c}^+$
is of dimension $12$.

We shall choose
the symmetric subpair $(G', K')
=(SO_0(2, 10), {SO}(2)
\times SO(10)
)$ of tube
type obtained from 
the  $V_2$ space in
(\ref{V-2-1-0}), 
with $D'=G'/K'=D\cap V_2$, with
Harish-Chandra strongly orthogonal roots $\gamma_1<\gamma_2$. Lemma \ref{root-sps}
is also verified in this case by counting dimensions.

The scalar
holomorphic representation $L(x\zeta)$ has Wallach set
$x\in (-\infty, 8)\cup \{8, 4, 0\}$.

Let  $\lambda=-k\gamma_1 +x\zeta=\lambda_0
+z\zeta$. The root system $Q(\lambda_0)=R(\lambda_0)=\fso(2, 10)$.
There are two reduction points $z=5, 9$,  and at the last
point $z=9, x=-8$,  and  $\lambda=-k\gamma_1 -8\zeta$.
The decomposition (\ref{ten-prod}) is
$L(-4\zeta)
\otimes 
L(-4\zeta)=\sum_{k=0}^\infty L(-8\zeta-k\gamma_1). $
Our theorem
in this case is 
also stated in \cite[Lemma 13.10]{EHW}
without proof.

%%\begin{proposition} The $K$-types in $L(\lambda)$, $\lambda=-8\zeta -k\gamma_1$ are precisely of the form$$
%%F(\lambda -m_1\gamma_1-m_2\gamma_2 -j\alpha_6), 0\le j\le \min\{m_2, k\}$$
%%and they appear with multiplicity one.  
%\end{proposition}

\smallskip

 {\it Proof of Theorem
\ref{thm1.1} for $(G,K)=(E_{7(-25)}, E_6\times SO(2)
)$.}
There is  a known formula \cite[Corollary 2.16]{St} 
computing the
tensor product
$
F(\nu)
\otimes
F(l\gamma)
$
for any minuscule weight $\gamma$ for any
simple Lie algebra, and we apply the formula to our case;
the decomposition is also verified in \cite{Ki-Wy} by
using computer algebra for larger $m$ and $l$.
We recall first
that
the representation
$\fp^-$ of $\mathfrak k=\mathfrak e_6$ is
of highest weight
$-\gamma_1$,  
 the fundamental representation $\omega_6$ dual to
$\alpha_6$, 
and is known to be minuscule
\cite[Section 7, Part A]{St}. 
We claim 
that 
\begin{equation}
\label{e6-tensor}
F(-m\gamma_1)
\otimes
F(-l\gamma_1)
=\sum_{i+j\le l}
F(
-(m+l-i)\gamma_1
-i\gamma_2 -j\alpha_6)
\end{equation}
(This might
be also proved
by using repeatedly 
\cite[Lemma 4.5]{EHW}.)

In our notation the 
 result
\cite{St} 
states that
$$
F(-m\gamma_1)
\otimes
F(-l\gamma_1)
=\sum_{
  \xi_1\le
\dots  \le\xi_l
}
F(
-m\gamma_1
+\xi_1+
\dots + \xi_l)
$$
where the sum is over all $\xi_j\in W[-\gamma_1]$
in the orbit of $-\gamma_1$  of the Weyl group
of $E_6$ such that
$-m\gamma_1 
+\xi_1+
\dots + \xi_j$
are dominant, $1\le j\le l$.
Here $m\ge l\ge 1$, 
the ordering
$\xi_1\le \xi_2\dots \le \xi_l,
\, \xi_j \in W[-\gamma_1]
$ are defined in the sense of the Bruhat ordering, i.e.,
the ordering obtained from the Weyl group $W$.
The claim 
(\ref{e6-tensor}) 
is equivalent to 
that the
only possible chains
in the sum
are
$$
\underbrace{-\gamma_2=\cdots=-\gamma_2}_{i}
\lneq
\underbrace{ \eta =\cdots=\eta}
_{j}
\lneq
\underbrace{-\gamma_1 =\cdots=-\gamma_1}_{l-i-j}
$$
where $
\eta=-(\e_6-\e_4) 
=-(\gamma_1 +\alpha_6)$, $\alpha_6=\e_5-\e_4$ being
the unique compact simple root connected to $\gamma_1$
in the Dynkin diagram. 

To prove our claim
we organize
the roots in $\mathfrak p^+$,
i.e. the roots $\Delta_n^+$
as the strongly orthogonal
roots $\gamma_1,
\gamma_2,
\gamma_3,
$ and three groups
of roots
\begin{equation}
\label{gp-1}
\{\frac{\gamma_1 +\gamma_2}2
+\epsilon_{12}^j,
1\le j\le 8\}
=\{\epsilon_6 \pm \epsilon_i, 1\le i\le 4\},
\end{equation}
\begin{equation}
\label{gp-2}
\{\frac{\gamma_1 +\gamma_3}2
+\epsilon_{13}^j,
1\le j\le 8
\}
=
\{\frac 12
(\epsilon_8 -
\epsilon_7 +
\epsilon_6 +
\epsilon_5 +
\sum_1^4 (-1)^{\nu(i)}
\epsilon_i), 
\sum_1^4\nu(i) \text{\, odd}
\}
\end{equation}
\begin{equation}
\label{gp-3}
\{\frac{\gamma_2 +\gamma_3}2
+\epsilon_{23}^j, 1\le j\le 8
\}
=\{\frac 12
(\epsilon_8 -
\epsilon_7 +
\epsilon_6 -\epsilon_5 +
\sum_1^4 (-1)^{\nu(i)}
\epsilon_i), 
\sum_1^4\nu(i) \text{\,even}
\}
\end{equation}
and the compact positive roots
$\Delta_c^+$
as
$$
\frac{\gamma_2 -\gamma_1}2
+\epsilon_{12}^j,\ 
\frac{\gamma_3 -\gamma_1}2
+\epsilon_{13}^j,\ 
\frac{\gamma_3 -\gamma_2}2
+\epsilon_{23}^j, \ 1\le j\le 8
$$
along with $12$
roots
\begin{equation}
\label{gp-4}
\epsilon_i\pm \epsilon_j,
\ 4\ge i > j\ge 1;
\end{equation}
all linear forms  $\epsilon_{12}^j,
\epsilon_{13}^j,
\epsilon_{23}^j, 1\le j\le 8
$ and the $12$ roots
are orthogonal to 
to $\gamma_1,
\gamma_2,
\gamma_3$. See e.g. \cite{Schmid, Up-AJM}
for general results.
Now $\xi_1=-\gamma_1, -\gamma_2$
are possible choice
for the first root in the chain.
If the first
roots $\xi_1\ne -\gamma_1,
\xi_2\ne -\gamma_2
$, then $\xi_1 =
\eta=-(\e_6-\e_4)$.
Indeed, 
$\xi_1=-\gamma_3$ is
not eligible since
$-m\gamma_1+\xi_1$
is not dominant.
Now let $-\xi_1$  be in
the first 
group (\ref{gp-1}),
then $\xi_1=-\eta=-(\e_6-\e_4)$
by the dominance condition. 
If $-\xi_1$ is in the second group
(\ref{gp-2}) then $\xi_1$
has to be the root
$\xi_1=-
\frac 12 (\epsilon_8-\epsilon_7+\epsilon_6+
\epsilon_5 -\epsilon_4
 -\epsilon_3
  -\epsilon_2
   +\epsilon_1
)$ by the dominance condition
with respect to the roots
(\ref{gp-4}). 
But then it
is not dominant with respect
to $\epsilon_1-\epsilon_4$.
If  $-\xi_1$
in the last  group
(\ref{gp-3})
then 
$\xi_1=-
\frac 12 (\epsilon_8-\epsilon_7+\epsilon_6
-
\epsilon_5 -\epsilon_4
 -\epsilon_3
 -\epsilon_2
   -\epsilon_1
)$ by the dominance  condition with respect to 
(\ref{gp-4}). But then 
$-m\gamma_1 +\xi_1$ is not
dominant with respect to
the simple root $\alpha_1$
of $\fe_6$, as
$(-m\gamma_1 +\xi_1, \alpha_1)
=-1<0$.
So the chain
$\xi_1\le \xi_2\dots \le \xi_l$
starts  with $\xi_1$ being
one of $-\gamma_1, -\gamma_2, 
\eta$.
If $\xi_1=\eta$
then  it is also
$\xi_1=\eta=s_{\alpha_6}(-\gamma_1)$, the reflection
defined by the simple root $\alpha_6$ connected to $\gamma_1$. Thus
the chain has
to be $\xi_1=
\cdots =\xi_l=\eta$
or with $l-j$ many $\eta$'s
and $j$ many $-\gamma_1$'s, since $-\gamma_1$ is the
only larger weight.
If $\xi_1= -\gamma_2
=-(\epsilon_6  +\epsilon_5)$
then again by the same computation
the next different
elements in the chain
have to be $\eta=-(\epsilon_6-\epsilon_4)$ and ending
with the same or with $-\gamma_1$.
This proves
our claim.

  Finally, we  use the
    same method
as above    by considering
    the similar
    tensor product 
    decomposition
    (\ref{e6-tensor})
for the     subdomain 
    $\widetilde D$ for $
    (\widetilde\fg, \widetilde\fk)=
   (\fso(2, 10), 
   \fso(2)+\fsl(10))$ in $D$. The corresponding
   Dynkin diagram is
   \[
\begin{tikzcd}[arrows=-]
D_6{:} &[-2em] \stackrel{3}{\bullet}   \arrow[r] &
\stackrel{4}{\bullet}  \arrow[r] \arrow[d] &
\stackrel{5}{\bullet} 
\arrow[r] &\stackrel{6}{\bullet}  \arrow[r] &
\stackrel{7}{\circ}  \\
& &
\stackrel {\bullet}{2}& &&
\end{tikzcd}
\]
obtained from the 
the Dynkin diagram 
       $\fe_7 $
       with the compact root $\alpha_1$ being deleted.
We can now use Lemma
\ref{reduct-subdom}
and 
our main Theorem for $\mathfrak{so}(2, 10)$.
Indeed the same
decomposition
(\ref{e6-tensor})
holds for the case
$\mathfrak{so}(2, 10)$, and the rest
of the proof is exactly similar 
to the previous cases in Sections
\ref{so*}
and \ref{e-6-case}.
\qed

\end{document}